\documentclass[a4paper,reqno]{paper}

\usepackage{amsfonts,amsthm,url}
\usepackage{amsmath}
\usepackage{amssymb}
\usepackage{colortbl}
\usepackage{array}

\usepackage{pstricks}
\usepackage{subcaption}
\usepackage{graphicx}

\usepackage{geometry}
\definecolor{Green}{cmyk}{1.,0,1.,0} 
\definecolor{Blue}{cmyk}{1.,1.,0,0} 
\definecolor{Orange}{cmyk}{0,0.61,0.87,0} 
\definecolor{Red}{cmyk}{0,1.,1.,0} 
\definecolor{ForestGreen}{cmyk}{0.91,0,0.88,0.12} 
\definecolor{Plum}{cmyk}{0.50,1.,0,0} 
\definecolor{YellowOrange}{cmyk}{0,0.42,1.,0} 

\newtheorem{Th}{Theorem}[section]
\newtheorem{Lem}[Th]{Lemma}
\newtheorem{Cor}[Th]{Corollary}
\newtheorem{Prop}[Th]{Proposition}

\theoremstyle{definition}
\newtheorem{Def}[Th]{Definition}
\newtheorem{Exa}[Th]{Example}

\theoremstyle{remark}
\newtheorem{Rem}[Th]{Remark}

\newcommand{\cA}{\mathcal{A}}
\newcommand{\cC}{\mathcal{C}}
\newcommand{\cF}{\mathcal{F}}
\newcommand{\cH}{\mathcal{H}}
\newcommand{\cZ}{\mathcal{Z}}

\newcommand{\sfs}{{\sf s}}
\newcommand{\sh}{{\sf h}}
\newcommand{\sq}{{\sf q}}
\newcommand{\bq}{\mathsf{q}}

\newcommand{\sH}{{\sf H}}

\newcommand{\nC}{\mathbf{C}}
\newcommand{\nN}{\mathbf{N}}
\newcommand{\nR}{\mathbf{R}}
\newcommand{\nZ}{\mathbf{Z}}

\newcommand{\bh}{\mathbf{h}}

\newcommand{\fI}{\mathfrak{I}}
\newcommand{\fJ}{\mathfrak{J}}

\newcommand{\id}{\textsf{id}}

\newcommand{\tA}{\tilde{A}}
\newcommand{\tC}{\tilde{C}}
\newcommand{\tG}{\tilde{G}}

\newcommand{\lra}{\longrightarrow}
\newcommand{\eq}{\Longleftrightarrow}
\newcommand{\lmt}{\longmapsto}

\newcommand{\ben}{\begin{enumerate}}
\newcommand{\een}{\end{enumerate}}

\newcommand{\bit}{\begin{itemize}}
\newcommand{\eit}{\end{itemize}}

\newcommand{\barr}{\begin{array}{cccccccccc}}
\newcommand{\ear}{\end{array}}

\newcommand{\tbu}{{\tiny $\bullet$}}

\newenvironment{maliste}%
{ \begin{list}%
	{\tbu}%
	{\setlength{\labelwidth}{30pt}%
	 \setlength{\leftmargin}{20pt}%
	 \setlength{\itemsep}{.05cm}}}%
{ \end{list} }

\newcommand{\bem}{\begin{maliste}}
\newcommand{\eem}{\end{maliste}}

\newcommand{\ov}{\overline}
\newcommand{\un}{\underline}

\newcommand{\sg}{\langle}
\newcommand{\sd}{\rangle}
\newcommand{\scal}[2]{\sg #1,#2 \sd} 

\newcommand{\alc}{\text{Alc}}

\newcommand{\quand}{\quad\text{and}\quad}

\newcommand{\ta}[3]{t^{#1}_\tau[{\scriptstyle{ #2, #3}}]} 
\newcommand{\tam}[3]{t^{#1}_{-\tau}[{\scriptstyle{ #2, #3}}]} 

\newcommand{\explain}[2]{\underset{#1}{\underbrace{#2}}}

\newcommand{\wt}{\text{wt}}
\newcommand{\stab}{\text{stab}}

\newcommand{\al}{\alpha}
\newcommand{\be}{\beta}
\newcommand{\de}{\delta}
\newcommand{\eps}{\epsilon}
\newcommand{\ga}{\gamma}
\newcommand{\la}{\lambda}
\newcommand{\om}{\omega}
\newcommand{\si}{\sigma}

\newcommand{\tla}{w_0t_{\la}}
\newcommand{\tlas}{t_{\la^{\hspace{-.2mm}\ast}}w_0}

\newcommand{\De}{\Delta}

\begin{document}

\parindent=0mm


\title{Admissible subsets and Littelmann paths in affine Kazhdan-Lusztig theory}

\author{J\'er\'emie Guilhot }

\maketitle

\begin{abstract} 
The center of an extended affine Hecke algebra is known to be isomorphic to the ring of symmetric functions associated to the underlying finite Weyl group $W_0$. The set of Weyl characters $\sfs_\la$ forms a basis of the center and Lusztig showed in~\cite{Lus15} that these characters act as translations on the Kazhdan-Lusztig basis element $C_{w_0}$ where $w_0$ is the longest element of $W_0$, that is we have $C_{w_0}\sfs_\la =C_{w_0t_\la}$. As a consequence, the coefficients that appear when decomposing~$C_{\tla}\sfs_\tau$ in the Kazhdan-Lusztig basis are tensor multiplicities of the Lie algebra with Weyl group $W_0$. The aim of this paper is to explain how admissible subsets and Littelmann paths, which are models to compute such multiplicities, naturally appear when working out this decomposition.
\end{abstract}
%
%
%
\section{Introduction}

Let $W_e$ be an extended affine Weyl group with underlying finite Weyl group $W_0$. Then $W_e=W_0\ltimes P$ where $P$ denotes the set of weights associated to $W_0$. Let~$\cH$ be the generic affine Hecke algebra of $W_e$  defined over~$\cA$ the ring  of Laurent polynomials with one indeterminate $\bq$ and let  $\{C_w\mid w\in W_e\}$ be the Kazhdan-Lusztig basis of $\cH$. The center of the affine Hecke algebra $\cH$ associated to $W_e$  is known to be isomorphic to the ring of symmetric functions $\cA[P]^{W_0}$. The set of Weyl characters $\{\sfs_\la\mid \la\in P^+\}$ forms a basis of $\cA[P]^{W_0}$ and we have~$C_{w_0}\sfs_\la=C_{w_0t_\la}$ where $t_\la$ denotes the translation by $\la\in P^+$ in $W_e$; see \cite{Lus15,R-N:03} and the references therein.  

\medskip

Denote by $V(\tau)$  the irreducible highest weight module of weight $\tau\in P^+$ for the simple Lie algebra over $\nC$ with Weyl group $W_0$ and weight lattice $P$. Then the character of $V(\la)$ is $\sfs_\la$ and for all $\tau,\la\in P^+$  we have~$\sfs_\la\sfs_\tau=\sum m_{\la,\tau}^\mu \sfs_\mu$ where $m_{\tau,\la}^\mu$ is the multiplicity of $V(\mu)$ in the tensor product $V(\la)\otimes V(\tau)$. Computing the multiplicities $m_{\tau,\la}^\mu$ is one of the most basic question in representation theory of simple  Lie algebras over~$\nC$. 
Littelmann showed \cite{Lit1} that such multiplicities can be determined by counting certain kind of paths  in the weight lattice $P$ constrained to stay in the fundamental chamber. 
Later on, Lenart and Postnikov \cite{LP,LP2} showed that these multiplicities can be determined using admissible subsets associated to a fix reduced expression of~$t_\tau\in W_e$. In \cite{LP} they used a geometric approach based on equivariant $K$-theory while in \cite{LP2} their approach  is more axiomatic and is based on the fact that their model satisfies a set of axioms, introduced by Stembridge in~\cite{St:02}, that encode the combinatorics of  Weyl characters. The model of Lenart and Postnikov can be viewed as a discrete counterpart of Littelmann paths model and they explicitly constructed a bijection between admissible subsets and Lakshmibai-Seshadri paths (which are certain kind of Littelmann paths).  

\medskip

In the extended affine Hecke algebra, we must have 
$$C_{w_0t_\la}\sfs_\tau=C_{w_0}\sfs_\la\sfs_{\tau}=\sum m_{\la,\tau }^\mu C_{w_0t_\mu}.$$
The aim of this paper is to explain how admissible subsets (and thus Littelmann paths) naturally appear when decomposing $C_{w_0t_\la}\sfs_\tau$ in the Kazhdan-Lusztig basis. Ultimately, this will be a consequence of a multiplication formula for two standard basis elements  in the extended affine Hecke algebra \cite[proof of Proposition 5.1]{bremke}. More precisely, we will, for all $\la,\tau\in P^+$
\ben
\item construct elements 
$\sh_{\tau}\in T_{t_\tau}+\sum_{y<t_\tau}\sq^{-1}\nZ[\sq^{-1}] T_y$
such that $C_{w_0}\sh_\tau=C_{w_0t_\tau}= C_{w_0}\sfs_\tau$;
\item show that $ C_{w_0t_\la}\sh_\tau$ is a $\nZ$-linear combination of Kazhdan-Lusztig basis elements;
\item show that determining the expansion in (2) is equivalent to finding terms of maximal degree in products of the form~$T_{ w_0t_\la v}T_{t_\tau}$ ($v\in W_0$) expressed in the standard basis;
\item show that the maximal terms in (3) are indexed by admissible subsets $J$ associated to a reduced expression of~$t_\tau$ as defined by Lenart and Postnikov.  
\een

The paper is organised as follows.  In Section 2, we introduce all the needed material on (extended) affine Weyl groups.
In Section 3, we present  Kazhdan-Lusztig theory for affine Hecke algebras with unequal parameters and we describe the center of these algebras. 
In Section 4 we prove (1)--(3) above: this will essentially be a consequence of results on the lowest two-sided cells in  \cite{guilhot7}. We will prove Statement (4) in Section 5. In Section~6, following \cite{LP2}, we study the connections between our work and the results of Lenart and Postnikov and we describe the bijection between admissible subsets and Lakshmibai-Seshadri paths.
%

\section{Affine Weyl groups}

Let $V$ be an Euclidean space  with scalar product $(\cdot,\cdot)$. We denote by 
$V^\ast$ the dual of $V$ and by $\langle\ ,\ \rangle : V \times V^* \longrightarrow \nR$  the canonical pairing. Let $\Phi$ be a root system and let $\Phi^\vee$ be the dual root system. If $\al\in \Phi$ then $\al^\vee\in \Phi^\vee$ is defined by~$\langle x,\al^{\vee} \rangle = 2 (x,\al)/(\al,\al)$. We fix a set of positive roots $\Phi^+$ and a simple system $\De=\{\al_1,\ldots,\al_N\}$ such that $\De\subset \Phi^+$.


\subsection{Geometric presentation of an affine Weyl group}
\label{geometric}
We denote by $H_{\al,n}$ the hyperplane defined by the equation $\sg x,\al^\vee \sd=n$ and by $\cF$ the collection of all such hyperplanes.
We will say that an hyperplane $H$ is of direction $\al\in\Phi^+$ if there exists a pair $(\al,n)\in \Phi^+\times \nZ$ such that $H=H_{\al,n}$. We then  write $\ov{H}=\al$. For any subset $F\subset \cF$ we set $\ov{F}:=\{\ov{H}\mid H\in F\}\subset \Phi^+$. Playing with notations yields $\ov{\{H\}}=\{\ov{H}\}$.

\medskip

 Let $W_a$ be the group generated by the set of orthogonal reflections $s_{\al,n}$ with respect to $H_{\al,n}$ where $\al\in \Phi$ and~$n\in \nZ$. The group $W_a$ is an affine Weyl group of type $\Phi^\vee$ and it is isomorphic to $W_0\ltimes Q$ where $Q$ is the lattice generated by $\Phi$. For $\la\in Q$, we will denote by $t_\la$  the translation by $\la$ in $W_a$. It is well known that~$W_a$ is generated by the set $S:=\{s_{\al_i,0}\mid \al\in \De\}\cup \{s_{\tilde{\al},1}\}$ where $\tilde{\al}^\vee$ is the highest root of $\Phi^\vee$.  
We will simply write~$s_{\al_i}$ for~$s_{\al_i,0}$ where $1\leq i\leq N$ and $s_{\al_0}$ for  $s_{\tilde{\al},1}$. 
  Let~$W_0$ be the stabiliser of $0$ in $W_a$ and  $w_0$ be the longest element of $W_0$.  Clearly $W_0=\sg S_0 \sd$ where $S_0=\{s_{\al_1},\ldots,s_{\al_N}\}$. We will denote by $\id$ the identity element in $W_a$.

\medskip

The set of alcoves, denoted $\alc(\cF)$, is the set of connected components of $V\backslash \cF$. The fundamental alcove $A_0$ is defined by
\begin{align*}
A_0&=\{x\in V\mid 0<\sg x,\al^\vee \sd<1,\ \forall \al\in \Phi^+\}\\
&=\{x\in V\mid 0<\sg x,\al^\vee \sd<1,\ \forall \al\in \De\}.
\end{align*}
The set of hyperplanes bounding $A_0$ is equal to $\{H_{\al,0}\mid \al\in \De\}\cup\{H_{\tilde{\al},1}\}$: these are called the walls of $A_0$. A face of $A_0$ (that is a codimension 1 facet) is said to be of type $s_{\al_i}$ if it is contained in the hyperplane $H_{\al_i,0}$ and of type $s_{\al_0}$
 if it is contained in the hyperplane $H_{\tilde{\al},1}$. 
 
 \medskip

The group $W_a$ acts simply transitively on the set of alcoves and we can extend the definition of walls and faces to all alcoves. Two alcoves $A$ and $A'$ are then said to be $s$-adjacent where $s\in S$ if they share a face $f$ of type $s$ and we write $A\sim_{s} A'$. In other words, $A$ and $A'$ are $s$-adjacent if there exists $w\in W_a$ such that $wf$ is the face of type $s$ of $A_0$.  
We will denote by $A_y$ the alcove $yA_0$.

\begin{Rem}
It is important to notice that the alcoves $A_w$ and $A_{ws}$ are $s$-adjacent for all $w\in W_a$ and all~$s\in S$. Therefore, any  expression $\vec{w}=s_{1}\ldots s_{n}$ ($s_i\in S$) of $w\in W_a$ defines a sequence of adjacent alcoves starting at the fundamental alcove $A_0$ and finishing at the alcove $A_w$:
$$e\sim_{s_{1}} A_{s_{1}}\sim_{s_{2}} A_{s_{1}s_{2}}\sim_{s_{3}}\ldots \sim_{s_{n}} A_{s_{1}\ldots s_{n}}=A_w.$$
Further, there exists a unique pair $(\al,k)\in \Phi^+\times \nZ$ such that the hyperplane $H_{\al,k}$ separates the alcoves $A_w$ and~$A_{ws}$ and we have $s_{\al,k}A_w=A_{ws}$. 
\end{Rem}

\medskip

Any hyperplane $H_{\al,n}$ where $\al\in \Phi^+$ divides the space $V$ into two half spaces 
$$H^+_{\al,n}=\{\la\in V\mid \scal{\la}{\al^\vee}>n\}\quand H^-_{\al,n}=\{\la\in V\mid \scal{\la}{\al^\vee}<n\}.$$
We say that an hyperplane $H\in \cF$ separates the alcoves $A$ and $B$ if and only if $A\in H^\eps$ and $B\in H^{-\eps}$  where~$\eps=\pm$. 
Given two alcoves $A,B\in \alc(\cF)$, we set 
$$H(A,B)=\{H\in \cF\mid H\text{ separates $A$ and $B$}\}.$$ 
For all $A\in\alc(\cF)$, $\al\in \Phi$ and $n\in \nZ$, we write $n<A[\al]$ (respectively $n>A[\al]$) if and only if for all $\la\in A$ we have~$n<\scal{\la}{\al^\vee}$ (respectively $n>\scal{\la}{\al^\vee}$). For two alcoves $A$ and $A'$, we write $A[\al]<A'[\al]$ if and only if  $\scal{\la}{\al^\vee}<\scal{\mu}{\al^\vee}$ for all $(\la,\mu)\in A\times A'$. For all alcoves $A$ and all roots~$\al\in \Phi$, there exists a unique $n\in \nZ$ such that $n<A[\al]<n+1$.

\medskip

The following proposition gathers some well known results about the length function and the action of $W_a$ on the set of alcoves. A standard reference for these results is \cite{Hum}.

\begin{Prop}
\label{basics}
Let $w\in W_a$. We have 
\ben
\item $\ell(w)=|H(A_0,A_w)|$
\item Let $s\in S$ and let $H_{\al,n}$ where $\al\in \Phi^+$ be the unique hyperplane separating $w$ and $ws$. We have $ws<w$ if and only  one of the following statement holds: 
\ben
\item $A_w\in H_{\al,n}^+$, $A_{ws}\in H_{\al,n}^-$ and $n>0$,
\item $A_w\in H_{\al,n}^-$, $A_{ws}\in H_{\al,n}^+$ and $n\leq 0$.
\een
\item We have $\ov{H(A_0,A_v)}=\{\al\in \Phi^+\mid v^{-1}\al\in \Phi^-\}$ for all $v\in W_0$.
\een
\end{Prop}

\begin{Exa}
\label{Exa-G2-1}
Let $\Phi$ be a root system of type $G_2$ and let $\Delta=\{\al_1,\al_2\}$ be a simple system in $\Phi$ such that $\al_2$ is the short root so that
$$\Phi=\pm\{\al_1,\al_2,\al_1+\al_2,\al_1+2\al_2,\al_1+3\al_2,2\al_1+3\al_2\}.$$
Then $\Phi^\vee$ is also of type $G_2$ and the coroot of $\al_1+2\al_2$ is $2\al_1^\vee+3\al_2^\vee$ which is the highest root of $\Phi^\vee$. 
In Figure \ref{rootG2}, we represent the root system $\Phi$. In Figure \ref{action}, we represent the alcoves $A_0$, $A_{w_0}$ and the alcove $A_w$ where $w=s_{\al_0}s_{\al_1}s_{\al_2}s_{\al_1}s_{\al_2}s_{\al_0}$. 

\psset{unit=.8cm}
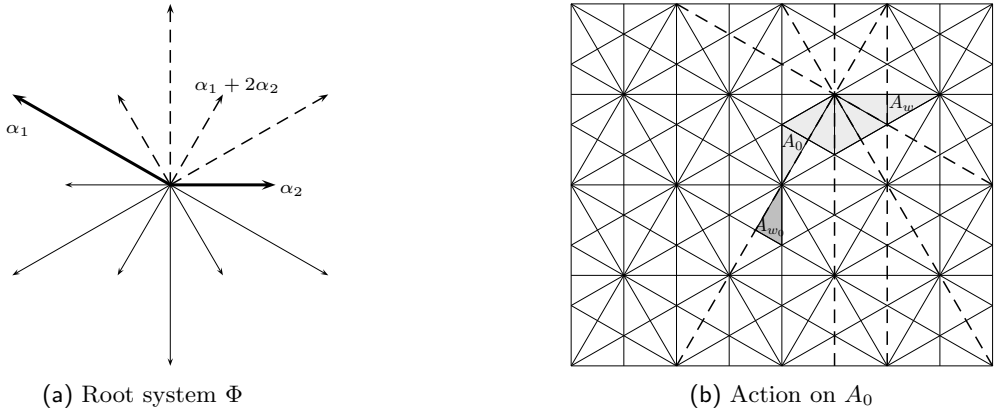
\begin{figure}[h]
\centering
\begin{subfigure}[b]{0.35\textwidth}
\centering
\begin{pspicture}(-4,-3)(4,3)
\psset{linewidth=.1mm}
\psline[linewidth=.4mm]{->}(0,0)(1.732,0)
\psline{->}(0,0)(-1.732,0)

\psline[linewidth=.2mm,linestyle=dashed]{->}(0,0)(2.598,1.5)
\psline{->}(0,0)(-2.598,-1.5)

\psline[linewidth=.4mm]{->}(0,0)(-2.598,1.5)
\psline{->}(0,0)(2.598,-1.5)

\psline[linewidth=.2mm,linestyle=dashed]{->}(0,0)(0,3)
\psline{->}(0,0)(0,-3)

\psline[linewidth=.2mm,linestyle=dashed]{->}(0,0)(.866,1.5)
\psline{->}(0,0)(.866,-1.5)
\psline[linewidth=.2mm,linestyle=dashed]{->}(0,0)(-.866,1.5)
\psline{->}(0,0)(-.866,-1.5)

\rput(2,-.1){{\scriptsize $\al_2$}}

\rput(1.1,1.65){{\scriptsize $\al_1+2\al_2$}}
\rput(-2.5,.9){{\scriptsize $\al_1$}}

\end{pspicture}\caption{Root system $\Phi$}
\label{rootG2}
\end{subfigure}
\qquad\qquad
\begin{subfigure}[b]{0.5\textwidth}
\centering
 \begin{pspicture}(-4,-3)(4,3)
                \psset{linewidth=.1mm}
\pspolygon[fillstyle=solid,fillcolor=gray!15!](0,0)(0,1)(.433,.75)

\pspolygon[fillstyle=solid,fillcolor=lightgray](0,0)(0,-1)(-.433,-.75)

\pspolygon[fillstyle=solid,fillcolor=gray!15!](1.732,1.5)(2.598,1.5)(1.732,1)

\pspolygon[fillstyle=solid,fillcolor=gray!15!](0,1)(.433,.75)(.866,1.5)
\pspolygon[fillstyle=solid,fillcolor=gray!15!](.433,.75)(.866,1.5)(.866,.5)
\pspolygon[fillstyle=solid,fillcolor=gray!15!](.866,1.5)(.866,.5)(1.3,.75)
\pspolygon[fillstyle=solid,fillcolor=gray!15!](.866,1.5)(1.3,.75)(1.732,1)
\pspolygon[fillstyle=solid,fillcolor=gray!15!](.866,1.5)(1.732,1)(1.732,1.5)

\psline(3.464,3)(-3.464,3)
\psline(3.464,-3)(-3.464,-3)
\psline(3.464,3)(3.464,-3)
\psline(-3.464,3)(-3.464,-3)

\psline(3.464,1.5)(-3.464,1.5)
\psline(3.464,0)(-3.464,0)
\psline(3.464,-1.5)(-3.464,-1.5)

\psline(2.598,3)(2.598,-3)
\psline[linewidth=.2mm,linestyle=dashed](1.732,3)(1.732,-3)
\psline[linewidth=.2mm,linestyle=dashed](.866,3)(.866,-3)
\psline(0,3)(0,-3)
\psline(-2.598,3)(-2.598,-3)
\psline(-1.732,3)(-1.732,-3)
\psline(-.866,3)(-.866,-3)

\psline(1.732,3)(3.464,2)
\psline(0,3)(3.464,1)
\psline[linewidth=.2mm,linestyle=dashed](-1.732,3)(3.464,0)
\psline(-3.464,3)(3.464,-1)
\psline(-3.464,2)(3.464,-2)
\psline(-3.464,1)(3.464,-3)
\psline(-3.464,0)(1.732,-3)
\psline(-3.464,-1)(0,-3)
\psline(-3.464,-2)(-1.732,-3)

\psline(-1.732,3)(-3.464,2)
\psline(0,3)(-3.464,1)
\psline(1.732,3)(-3.464,0)
\psline(3.464,3)(-3.464,-1)
\psline(3.464,2)(-3.464,-2)
\psline(3.464,1)(-3.464,-3)
\psline(3.464,0)(-1.732,-3)
\psline(3.464,-1)(0,-3)
\psline(3.464,-2)(1.732,-3)

\psline(-3.464,0)(-1.732,-3)
\psline(-3.464,3)(0,-3)
\psline(-1.732,3)(1.732,-3)
\psline[linewidth=.2mm,linestyle=dashed](0,3)(3.464,-3)
\psline(1.732,3)(3.464,0)

\psline(3.464,0)(1.732,-3)
\psline(3.464,3)(0,-3)
\psline[linewidth=.2mm,linestyle=dashed](1.732,3)(-1.732,-3)
\psline(0,3)(-3.464,-3)
\psline(-1.732,3)(-3.464,0)

\rput(0.1,.7){ \scalebox{.7}{$A_0$}}
\rput(-0.27,-.7){ \scalebox{.7}{$A_{w_0}$}}
\rput(1.9,1.35){ \scalebox{.7}{$A_w$}}

\end{pspicture}
\caption{Action on $A_0$}
\label{action}
\end{subfigure}
 \caption{Roots, hyperplanes and alcoves in type $G_2$}
 \label{G2}
\end{figure}
The set of hyperplanes that separates $A_0$ and $A_w$ 
is
$$H(A_0,A_w)=\left\{H_{\al_2,2},H_{\al_1+2\al_2,2},H_{\al_1+3\al_2,1},H_{\al_2,1},H_{\al_1,0},H_{\al_0,1}\right\}$$ 
and these are represented with dashed lines. Note that the cardinal of $H(A_0,wA_0)$ is indeed the length of~$w$. The alcoves of the sequence 
$$(A_0,A_{s_0},A_{s_0s_1},A_{s_0s_1s_2},A_{s_0s_1s_2s_1},A_{s_0s_1s_2s_1s_2},A_{w})$$ are colored in light gray. Finally we have
$$s_{\al_0}s_{\al_1}s_{\al_2}s_{\al_1}s_{\al_2}s_{\al_0}A_0=s_{\al_2,2}s_{\al_1+2\al_2,2} s_{\al_1+3\al_2,1}s_{\al_2,1}s_{\al_1,0}s_{\al_0,1}A_0.$$

\end{Exa}


\subsection{Weight functions and special points}
\label{Lpoints}
Let $L$ be a positive weight function on $W_a$, that is a function $L:W_a\lra \nN$ such that $L(ww')=L(w)+L(w')$ whenever $\ell(ww')=\ell(w)+\ell(w')$. To determine a weight function, it is enough to give its values on the conjugacy classes of generators of $S$. From now on, we fix such a positive weight function $L$ on $W_a$. 

\medskip

Let $H$ be an hyperplane in $\cF$. We say that $H$ is of weight $L(s)$ if it contains a face of type $s\in S$.  This is well-defined since if $H$ contains a face of type $s$ and $s'$ then $s$ and $s'$ are conjugate and $L(s)=L(s')$ see \cite[Lemma 2.1]{bremke}. We denote the weight of an hyperplane $H$ by $L(H)$.  If $\al\in \Phi$, we set $L(\al)=\max_{\ov{H}=\al} L(H)$.
 For any~$\la\in V$ we set
$$L(\la)=\sum_{H,\la\in H} L(H).$$
Let $\nu=\max_{\la\in V} L(\la)$. We call $\la$ an $L$-weight if $L(\la)=\nu$ and we denote by $P$ the set of $L$-weights. Further we denote by $P^+$ the set of dominant $L$-weights that is $P^+=\{\la\in P\mid \scal{\la}{\al^\vee}\geq 0 \text{ for all } \al\in \Phi^+\}$ and by $P^-$ the set of anti-dominant weights, that is $P^-=-P^+$. Without loss of generality, we will always assume that~$0\in V$ is an~$L$-weight.  The action of the longest element $w_0\in W_0$ on the set of weights is an involution that sends $P^+$ onto $P^-$ and  we set $\la^\ast=w_0\la$ for all~$\la\in P$.


\subsection{Extended affine Weyl groups}
The extended affine Weyl group is defined by $W_e=W_0\ltimes P$; it acts naturally on the set of alcoves $\alc(\cF)$ but the action is no longer faithful. If we denote by $\Pi$ the stabiliser of $A_0$ in $W_e$ then we have $W_e=\Pi\ltimes W_a$. Note that $\Pi$ permutes the weight that belong to the closure of $A_0$. Further the group $\Pi$ is isomorphic to $P/Q$, hence it is abelian and its action on  $W_a$ is given by an automorphism of the Dynkin diagram; see Planches I--IX in~\cite{bourbaki}. We denote $t_\la$ where $\la\in P$ the translation by $\la$ in $W_e$.

\medskip

An extended alcove is a pair $(A,\mu)$ where $A\in\alc(\cF)$ and $\mu$ is a vertex of $A$ which lies in $P$. We denote by~$\alc_e(\cF)$ the set of extended alcoves. The group $W_e$ acts naturally on $\alc_{e}(\cF)$ and the action is faithfull and transitive.
 Indeed, if $(A,\mu)\in \alc_e(\cF)$, there exists $w\in W_a$ such that  $wA=A_0$ and $w\mu=\mu'$ where $\mu'\in \ov{A_0}$. If we let $\pi\in \Pi$ be such that $\pi \mu'=0$ we obtain $\pi w(A,\mu)=(A_0,0)$ as required. 
To simplify the notation, we will simply write $A_0$ for $(A_0,0)$ and for all~$w\in W_e$ we set $A_{w}=wA_0$.

\medskip

All the notions and notations for alcoves in $\alc(\cF)$ can be extended to $\alc_e(\cF)$. We just omit the part  with the weight when needed. For instance if $A'=(A,\la)\in \alc_{e}(\cF)$, we write $A'[\al]<0$ to mean $A[\al]<0$. The length function, the weight function, the Bruhat order all naturally extend to $W_e$ by setting $\ell(aw)=\ell(w)$,~$L(aw)=L(w)$ and $aw<a'w'$ if and only if $a=a'$ and $w<w'$ where~$a,a'\in \Pi$ and $w,w'\in W_a$.


\newcommand{\quor}{\quad\text{or}\quad}

\subsection{Quarter of vertex $\la$}
\label{quarters}

The quarters of vertex $\la\in V$ are the connected components of 
$$V\backslash \bigcup_{H\in \cF,\la\in H} H.$$
Given $\la\in P$ and $v\in W_0$, we denote by  $\cC_{\la,v}$ the quarter of vertex $\la$ which contains $t_\la v$. When we consider a quarter with vertex $0$ we will omit the $0$ in the notation. The set of Weyl chambers is then $\{\cC_{w}\mid w\in W_0\}$ and the fundamental Weyl chamber is  $\cC_{\id}$. We have 
\begin{align*}
\cC_{\id}&:=\{x\in V\mid \scal{x}{\al^\vee}>0\text{ for all $\al\in \Phi^+$}\}\quand\\
\cC_{w_0}&:=\{x\in V\mid \scal{x}{\al^\vee}<0\text{ for all $\al\in \Phi^+$}\}.
\end{align*}
Let $X_0$  be the set of right coset representatives of minimal length of $W_0$  in $W_e$. Then $X_0$ is the set of $x\in W_e$ that satisfies $\ell(w_0x)=\ell(w_0)+\ell(x)$ and
$$x\in X_0\eq A_x\in \cC_{\id}\eq A_x[\al]>0 \text{ for all $\al\in \De$}.$$
 Any element $w$ of $W_e$ can be uniquely written under the form $w=vx$ where $v\in W_0$ and $x\in X_0$. For all~$x\in X_0$,~$\la\in P$ and $v\in W_0$, the alcove  $t_\la vx$ lies in~$\cC_{\la,v}$. 
  
\medskip

For $\la\in V$ and $\al\in \Phi^+$ we set $\la_{\al}=\scal{\la}{\al^\vee}$.  Let $\cC$ be a quarter of vertex $\la\in P$ and fix $\al\in \Phi^+$.  We have either 
$$\{\scal{x}{\al^\vee}\mid x\in \cC\}= ]\la_\al,+\infty[\quor\{\scal{x}{\al^\vee}\mid x\in \cC\}= ]-\infty,\la_\al[.$$
In the first case we say that $\cC$ is oriented toward $+\infty$ in the direction $\al$ and we write $\cC[\al]=+\infty$.
In the second case we say that $\cC$ is oriented toward $-\infty$ in the direction $\al$ and we write $\cC[\al]=-\infty$.  

\begin{Lem}
\label{pminfty}
Let $\la\in P$ and $v\in W_0$. 
We have for all $\al\in \Phi^+$ :
$$\cC_{\la,v}[\al]=\begin{cases}
+\infty &\mbox{ if $v^{-1}\al\in \Phi^+$,}\\
-\infty &\mbox{ if $v^{-1} \al \in \Phi^-$.}
\end{cases}$$
\end{Lem}
\begin{proof}
Let $x\in \cC_{\la,v}$. There exist $x_0\in \cC_{v}$ and  $y_0\in \cC_{\id}$  such that $x=x_0+\la$ and  $x_0=vy_0$.
We have 
$$\scal{x}{\al^\vee}=\scal{x_0+\la}{\al^\vee}=\scal{vy_0}{\al^\vee}+\la_\al=\scal{y_0}{(v^{-1}\al)^\vee}+\la_\al.$$
Since $y_0\in \cC_{\id}$ we have $\scal{y_0}{(v^{-1}\al)^\vee}\geq 0$ if and only if $v^{-1}\al\in \Phi^+$ as required. 
\end{proof}

%
%
%

\section{Affine Hecke algebra with unequal parameters}
\subsection{Affine Hecke algebra}
Let $\cA=\nC[\bq,\bq^{-1}]$ where $\bq$ is an indeterminate. The Iwahori-Hecke algebra $\cH$ associated to $W_e$ is the free~$\cA$-module with basis $(T_{w})_{w\in W_e}$ and relations given by 
$$T_uT_v=T_{uv} \text{ whenever $\ell(uv)=\ell(u)+\ell(v)$}$$
and
$$(T_s-\bq^{L(s)})(T_s+\bq^{-L(s)})=0 \text{ if $s\in S$.}$$
From this relation, we easily find that for all $s\in S$ and all $w\in W$, we have 
\begin{equation*}
T_{s}T_{w}=
\begin{cases}
T_{sw} & \mbox{if } \ell(sw)>\ell(w),\\
T_{sw}+\xi_sT_{w} &\mbox{if } \ell(sw)<\ell(w)
\end{cases} \quad\text{ where $\xi_s=\bq^{L(s)}-\bq^{-L(s)}$.}
\end{equation*}
The basis $(T_{w})_{w\in W_e}$ is called the standard basis. We write $f_{x,y,z}$ for the structure constants with respect to this basis:
$$T_xT_y=\sum_{z\in W_e} f_{x,y,z} T_z.$$
The elements $f_{x,y,z}$ are polynomials in $\{\xi_s\mid s\in S\}$ with positive coefficients. The degree of $f_{x,y,z}$ will be denoted $\deg(f_{x,y,z})$ and is the highest power of $\bq$ that appears in $f_{x,y,z}$.

\subsection{Multiplication of the standard basis}
\label{notation-standard}
 In this section, we present a result of \cite{guilhot2} on a bound on the degree of the polynomials $f_{x,y,z}$. 
Recall that for two alcoves  $A,B\in \alc_e(\cF)$, we have set 
$H(A,B)=\{H\in \cF\mid H \text{ separates } A \text{  and } B\}$. 
Then for $x,y\in W_e$ we set
\begin{align*}
\sH_{x,y}=H(A_0,A_x)\cap H(A_x,A_{xy})\quand c_{x,y}(\al):=\underset{\ov{H}=\al}{\max_{H\in \sH_{x,y}}} L_H. 
\end{align*}
Then according to \cite[Theorem 2.4]{guilhot2} we have:
\begin{Th}
\label{bound}
The degrees of the polynomials $f_{x,y,z}$ are bounded by 
$\sum_{\al\in \ov{\sH_{x,y}}} c_{x,y}(\al)$.
\end{Th}
We obtain the following corollary \cite[Proposition 5.3]{guilhot7} which will be crucial in the following section.

\begin{Cor}
\label{firstbound}
Let $x\in X^{-1}_0$, $v\in W_0$ and $y\in X_0$. We have
$$\deg(f_{xv,y,z})\leq L(w_0)-L(v).$$  
\end{Cor}


\subsection{Kazhdan-Lusztig basis}

Let $\bar\ $ be the ring involution of $\cA$ which takes $\bq$ to $\bq^{-1}$. This involution can be extended to a ring involution of $\cH$ via the formula
$$\ov{\sum_{w\in W_e} a_{w}T_{w}}=\sum_{w\in W_e} \ov{a}_{w}T_{w^{-1}}^{-1}\quad (a_{w}\in\cA).$$
We set 
$$\begin{array}{lllcclccccc}
\cA_{<0}=&\bq^{-1}\nZ[\bq^{-1}] &,& \cH_{<0}=\bigoplus_{w\in W_e} \cA_{<0}T_{w}\\[.4mm]
\cA_{\leq 0}=&\nZ[\bq^{-1}] &\text{and}& \cH_{\leq 0}=\bigoplus_{w\in W_e} \cA_{\leq 0}T_{w}.
\end{array}$$
For each $w\in W_e$ there exists a unique element $C_{w}\in\cH$ (see \cite[Theorem 5.2]{bible}) such that (1) $\overline{C}_{w}=C_{w}$ and (2)~$C_{w}\equiv T_{w} \mod \cH_{<0}$.
For any $w\in W_e$ we set 
$$C_{w}=\sum_{y\in W_e } P_{y,w} T_{y}\quad \text{where $P_{y,w}\in \cA_{< 0}$}.$$
The coefficients $P_{y,w}$ are called  the Kazhdan-Lusztig polynomials.  It is well known (\cite[\S 5.3]{bible}) that $P_{y,w}=0$ whenever $y\nleq w$ and that $P_{w,w}=1$. It follows that $(C_{w})_{w\in W_e}$ forms an $\cA$-basis of $\cH$ known as the Kazhdan-Lusztig basis. According to \cite[Theorem 6.6]{bible} we have
$$\forall w\in W_e, \forall s\in S, P_{x,y}=\bq^{-L(s)}P_{xs,y} \text{ whenever } x<xs \text{ and } ys<y.$$

\begin{Rem}
\label{firstbound2}
Using Corollary \ref{firstbound} and the definition of the Kazhdan-Lusztig basis, one can show that  the element $C_{w_0t_\la}T_y$ lies in $\cH_{\leq 0}$  for all $\la\in P^+$ and $y\in X_0$; see Proposition 5.3 in \cite{guilhot7}. Indeed we have 
\begin{align*}
C_{w_0t_\la}T_y&=\sum_{x\leq \tla} P_{x,w_0t_\la} T_x T_y\\
&=\sum_{x_0\in X_0, v\in W_0} P_{x_0^{-1}v,\tlas} T_{x_0^{-1}v} T_y\\
&=\sum_{ x_0^{-1}v\leq w_0t_\la} \bq^{L(v)-L(w_0)}P_{x_0^{-1}w_0,\tlas} T_{x_0^{-1}v} T_y
\end{align*}
The result follows since $\deg(f_{x_0^{-1}v,y,z})\leq L(w_0)-L(v)$ for all $z\in W_e$.
\end{Rem}

\subsection{The center of the affine Hecke algebra}
In this section we follow the presentation of Nelsen and Ram \cite{R-N:03} and we refer to it and the references therein for details and proofs.  We start by introducing another presentation of the affine Hecke algebra which is more convenient to describe its center $Z(\cH)$. For each $\la\in P$, we set $e^\la:=T_{t_\mu}T_{t_\nu}^{-1}$ where $\mu,\nu\in P^+$ are such that~$\mu-\nu=\la$. This can be shown to be independent of the choice of $\mu$ and $\nu$. Then $\cH$ is generated by the sets~$\{T_{s}\mid s\in S_0\}$ and~$\{e^\la\mid \la\in P\}$ and we have the relations
$$e^\la e^\mu=e^{\la+\mu}=e^\mu e^\la\quand e^\la T_{s_i}=T_s e^{s_i\la}+\xi_s \dfrac{e^\la-e^{s_i\la}}{1-e^{-\al_i}}.$$
Let $\cA[P]$ be the subalgebra of $\cH$ generated by $\{e^\la\mid \la\in P\}$. Then $W_0$ acts naturally on $\cA[P]$ via~$w\cdot e^\la=e^{w\la}$.

\begin{Th}
The sets $(e^\la T_w)_{\la\in P, w\in W_0}$ and $(T_w e^\la)_{w\in W_0, \la\in P}$ are $\cA$-basis of $\cH$ and the center of $\cH$ is 
$$\cZ(\cH)=\{f\in \cA[P]\mid  w\cdot f=f \text{ for all $w\in W_0$}\}.$$
\end{Th}
We define the Weyl characters $\sfs_\la$ as follows
$$a_\la:=\sum_{w\in W_0}(-1)^{\ell(w)}e^{w\la}\quand \sfs_{\la}=\dfrac{a_{\la+\rho}}{a_{\rho}} \quad \text{ where } \rho=\dfrac{1}{2}\sum_{\al\in \Phi^+} \al.$$
 Then $(\sfs_{\la})_{\la\in P^+}$ form a basis of $\cZ(\cH)$. According to  \cite[Theorem 2.9]{R-N:03}, the element~$\sfs_{\la}$ acts as translation on the Kazhdan-Lusztig element $C_{w_0}$. 
 \begin{Th}
We have $\sfs_\la C_{w_0}=C_{w_0}\sfs_\la=C_{w_0t_\la}$ for all $\la\in P^+$ and 
$$C_{w_0t_\la}\sfs_{\tau}=C_{w_0}\sfs_{\la}\sfs_{\tau}=\sum_{\mu} m^\mu_{\la,\tau}C_{w_0t_\mu}.$$
\end{Th} 

%

\section{Decomposition into the Kazhdan-Lusztig basis}
The aim of this section is to show that in order to determine the decomposition of  $C_{w_0t_\la}\sfs_{\tau}$ in the Kazhdan-Lusztig basis, it is enough to determine the terms of maximal degree in the products $T_{w_0t_{\la}v}T_{t_\tau}$ where $v\in W_0$. This will be done in 3 steps:
\ben
\item we construct some special elements $\sh_\tau$ such that  $C_{w_0t_\la}\sfs_{\tau}=C_{w_0t_\la}\sh_{\tau}$ and
${\sh}_{\tau}\in T_{t_\tau}+\underset{y<t_\tau,y\in X_0}{\bigoplus}\cA_{<0}T_y,$ 
\item we show that
$C_{w_0t_\la}\sh_{\tau}\equiv \sum a_x T_x \mod \cH_{<0}$ where $a_x\in \nZ$,
\item we show that  $\ov{C_{w_0t_\la}\sh_{\tau}}=C_{w_0t_\la}\sh_{\tau}$  and therefore $C_{w_0t_{\la}}\sh_\tau =\sum a_x C_x$.
\een

Let  $x\in X_0$ and set $\sh_x:=\sum_{x'\in X_0} P_{w_0x',w_0x} T_{x'}$. Following \cite[Lemma 2.5]{Xie:17} we get  
\begin{align*}
C_{w_0x}&=\sum_{y\in W_e} P_{y,w_0x} T_{y}\\
&=\sum_{x'\in X_0,u\in W_0} P_{ux',w_0x} T_{ux'}\\
&=\sum_{x'\in X_0,u\in W_0} \bq^{L(u)-L(w_0)}P_{w_0x',w_0x} T_{u}T_{x'}\\
&=\left( \sum_{u\in W_0} \bq^{L(u)-L(w_0)}T_{u}\right)\sum_{x'\in X_0} P_{w_0x',w_0x}T_{x'} \\
&=C_{w_0}{\sh}_x.
\end{align*}
When $\tau\in P^+$, we know that $t_\tau\in X_0$ and we will write ${\sh}_\tau$ instead of ${\sh}_{t_\tau}$. For all $\tau,\la\in P^+$ we have 
$$ C_{w_0}\sfs_\tau=C_{w_0}\sh_\tau\quand C_{w_0t_\la} \sfs_{\tau}=\sfs_{\la}C_{w_0}\sfs_\tau=\sfs_{\la}C_{w_0}\sh_\tau=C_{w_0t_\la}
\sh_\tau.$$ 
This concludes part~(1) since the elements ${\sh}_\tau$ have the required form.  Similarly, we can construct elements  ${\mathsf g}_{x}$ for all $x\in X^{-1}_0$ such that ${\mathsf g}_xC_{w_0}=C_{xw_0}$. Setting ${\sf g_{t_\tau}}={\sf g}_\tau$ ($\tau\in P^-$) we  have $C_{w_0}\sh_{\tau}=C_{w_0t_{\tau}}=C_{t_{\tau^\ast}w_0}={\sf g}_{\tau^\ast}C_{w_0}$.

\medskip

 Let $\tau,\la \in P^+$. Using Remark \ref{firstbound2} and the fact that $P_{w_0x, w_0t_\tau}\in \cA_{<0}$ whenever $x<t_\tau$, we get
\begin{align*}
C_{w_0 t_\la}\sh_{\tau}&=C_{w_0 t_\la}T_{t_\tau}+\sum_{x<t_\tau,x\in X_0} \underset{\in \cH_{<0}}{\underbrace{P_{w_0x, w_0t_\tau}  C_{w_0 t_\la}T_x}}\equiv C_{w_0 t_\la}T_{t_\tau} \mod \cH_{<0}
\end{align*}
and
\begin{align*}
C_{w_0 t_\la}T_{t_\tau}&=T_{w_0 t_\la}T_{t_\tau}+\sum_{y<w_0 t_\la} P_{y,w_0 t_\la} T_yT_{t_\tau} .
\end{align*}
Let $y\in W_e$ and $(y_r,y_0)\in X^{-1}_0\times W_0$ be such $y=y_ry_0$. On the one hand $P_{y,w_0 t_\la}=\bq^{L(y_0)-L(w_0)}P_{y_rw_0,\tlas}$ and on  the other hand, by Corollary \ref{firstbound},  the maximal degree that can appear in $T_{y_ry_0}T_{t_\tau}$ is $L(w_0)-L(y_0)$. Therefore if $y_rw_0<\tlas$ we get that $P_{y,w_0t_\la} T_{y}T_{t_\tau}\in \cH_{<0}$ and
\begin{align*}
C_{w_0 t_\la}T_{t_\tau}&\equiv \sum_{y_0\in W_0}  \bq^{L(y_0)-L(w_0)} T_{t^\ast_{\la}y_0}T_{t_\tau}\equiv \sum_{v\in W_0}  \bq^{-L(v)}T_{t^\ast_{\la}w_0v}T_{t_\tau} \mod \cH_{<0}.
\end{align*}
Finally
\begin{align}
\label{firstapprox}
C_{w_0t_\la}{\sh}_\tau &\equiv \sum_{v\in W_0}   \underset{\in \cH_{\leq 0}}{\underbrace{\bq^{-L(v)} T_{w_0t_{\la}v}T_{t_\tau}}}\mod \cH_{<0}
\end{align}
as claimed in statement (2). 

\medskip

We now prove Statement (3). 
\begin{Lem}
For all $\la,\tau\in P^+$, the elements  $C_{w_0t_\la}{\sh}_\tau$ are stable under the $\bar{\ }$-involution.
\end{Lem}
\begin{proof}
First we have $C_{w_0t_\la}{\sh}_\tau=C_{w_0}\sh_\la\sh_\tau={\sf g}_{\la^\ast}{\sf g}_{\tau^\ast}C_{w_0}$ so that $C_{w_0t_\la}{\sh}_\tau$ lies in the intersection of $\cH C_{w_0}$ and $C_{w_0}\cH$. Next 
$$\cH C_{w_0}=\sg C_{xw_0}\mid x\in X_0^{-1}\sd_\cA\quand \cH C_{w_0}=\sg C_{w_0x}\mid x\in X_0\sd_\cA$$
from where we get   
$$\cH C_{w_0}\cap C_{w_0}\cH=\sg C_{w_0 t_\nu}\mid \nu\in P^+ \sd_{\cA}$$
 since $X^{-1}_0w_0\cap w_0X_0=\{w_0t_{\nu}\mid \nu\in P^+\}$. Therefore there exist $b_\nu\in \cA$ such that $C_{w_0t_\la}{\sh}_\tau=\sum_{\nu\in P^+} b_\nu C_{w_0t_\nu}$. At this stage,  in order to show that $\ov{C_{w_0t_\la}{\sh}_\tau}=C_{w_0t_\la}{\sh}_\tau$ it is now enough to show that $b_\nu\in \nZ$ since $\ov{ C_{w_0t_\nu}}= C_{w_0t_\nu}$. The following argument is inspired by \cite[Proof of Theorem 6.2]{Xie:15}.  We have
\begin{align*}
h_{w_0,w_0,w_0}C_{w_0t_\la}\sh_{\tau}&= C_{t_{\la^\ast}w_0}C_{w_0}\sh_{\tau}\\
&=  C_{t_{\la^\ast}w_0}C_{w_0t_\tau}\\
&= \sum_{z\in W_e} h_{t_{\la^\ast}w_0,w_0t_\tau,z}C_z\\
&= \sum_{\nu\in P^+} b_{\nu}  h_{w_0,w_0,w_0}C_{w_0t_\nu}
\end{align*}
where $h_{x,y,z}\in \cA$ denote the structure constants with respect to the Kazhdan-Lusztig basis. 
According to \cite[\S 13.4]{bible}, the degree of $h_{x,y,z}$ is bounded by $L(w_0)$ for all $x,y,z\in W_e$. Therefore  $\deg(h_{\tau^{-1}w_0,w_0t_\tau,z})\leq L(w_0)$ and since $\deg(h_{w_0,w_0,w_0})=L(w_0)$, this forces~$b_{\nu}\in \nZ$ as required. 
\end{proof}
Statement (3) now follows using the following lemma. 
\begin{Lem}
\label{mod-H0} Let $h\in \cH$ be such that $\bar{h}=h$ and $h\equiv \sum a_xT_x \mod \cH_{<0}$ where $a_x\in \nZ$. Then $h=\sum a_x C_x$.
\end{Lem}
\begin{proof}
We know \cite[\S 5.2.(e)]{bible} that if $h'\in\cH_{<0}$ satisfies $\bar{h'}=h$ then $h'=0$. The lemma is an easy consequence of this result setting $h'=h-\sum a_x C_x$.   
\end{proof}
As a consequence, in order to determine the decomposition of $C_{w_0t_\la}{\sh}_\tau$ in the Kazhdan-Lusztig basis, we need to determine which product $\bq^{-L(v)} T_{w_0t_\la v}T_{t_\tau}$ can actually give rise to a non zero term modulo $\cH_{<0}$. In other words, we need to determine which terms in the decomposition of $T_{w_0t_\la v}T_{t_\tau}$ in the standard basis has a coefficient of (maximal) degree $L(v)$.

\medskip

The remainder of this section is devoted to set up the notation in order to study the product  $T_{\tla v}T_{t_\tau}$.
Let~$x,y\in W_e$ and let  $\vec{y}=s_1\ldots s_n a$ be a reduced expression of~$y$ where $a\in \Pi$ and $s_i\in S$ for all $i$. Let~$J=\{i_1,\ldots,i_p\}$ be a subset of $\{1,\ldots,n\}$.
For all $1\leq \ell,k< n$, we set 
$$\vec{y}^J=\left(\prod_{r=1, r\notin J}^n s_{r}\right) a, \quad \vec{y}^J{\scriptstyle [\ell,k]}=\prod_{r=\ell, r\notin J}^k s_{r}\quand \vec{y}^J{\scriptstyle [\ell,n]}=\left(\prod_{r=\ell, r\notin J}^n s_{r}\right) a.$$
When $J$ is empty we will simply write $\vec{y}{\scriptstyle [\ell,k]}$ instead of $\vec{y}^{\hspace{.1mm}\varnothing}{\scriptstyle [\ell,k]}$ for the product $s_\ell\ldots s_{k}$ and $\vec{y}{\scriptstyle [\ell,n]}$ instead of $\vec{y}^{\hspace{.08mm}\varnothing}{\scriptstyle [\ell,n]}$. Finally we denote by $p_J(x;\vec{y})$ the sequence of alcoves $(A_{x\vec{y}^J{\scriptstyle [1,k]}})_{1\leq k\leq n}$. We see that
\bem
\item there can be repetitions in this sequence (see below);
\item any two consecutive alcoves are either equal or adjacent.
\eem
We can therefore represent $p_J(x;\vec{y})$ by a path going through the sequence of alcoves $(A_{x\vec{y}^J\scriptstyle [1,k]})_{1\leq k\leq n}$ and which folds on the $s_i$-face of $A_{x\vec{y}^J{\scriptstyle [1,i-1]}}$ for all $i\in J$ (and hence goes twice through the alcove $A_{x\vec{y}^J{\scriptstyle [1,i-1]}}=A_{x\vec{y}^J{\scriptstyle [1,i]}}$). We will say that the path $p_J(x;\vec{y})$ is included in a certain subset of $V$ if all the alcoves that appear in $p_J(x;\vec{y})$ lie in this subset. 

\medskip

Let $\fI_{x,\vec{y}}$ the set of all subsets $\{i_1,\ldots,i_p\}$ of $\{1,\ldots,n\}$ such that $1\leq i_{1}<\ldots<i_{p}\leq n$ and
$$x\vec{y}^J{\scriptstyle [1,i_\ell-1]}s_{i_\ell}<x\vec{y}^J{\scriptstyle [1,i_\ell-1]} \text{ for all $\ell\in \{1,\ldots,p\}$}.$$
For $J=\{i_{1},\ldots,i_{p}\}$ in $\fI_{x,\vec{y}}$, we set $\xi_J=\prod^{p}_{k=1} \xi_{s_{i_k}}$ so that we have \cite[Proof of Proposition 5.1]{bremke}
$$T_{x}T_{y}=\sum_{J\in\fI_{x,\vec{y}}}\xi_J T_{x\vec{y}^J}.$$
Let $H_{\al,n}$ ($\al\in \Phi^+$) be the hyperplane that separates $A_{x\vec{y}^J{\scriptstyle [1,{i_\ell}-1]}}$ and $A_{x\vec{y}^J{\scriptstyle [1,{i_\ell}-1]}s_{i_{i_\ell}}}$. By definition of $\fJ_{x,\vec{y}}$ we have $x\vec{y}^J{\scriptstyle [1,{i_\ell}-1]} s_{i_\ell}<x\vec{y}^J{\scriptstyle [1,{i_\ell}-1]}$ and therefore 
\bem
\item $A_{x\vec{y}^J\scriptstyle [1,{i_\ell}-1]}\in H^-_{\al,n}$ and $A_{x\vec{y}^J\scriptstyle [1,{i_\ell}-1]}s_{{i_\ell}}\in H^+_{\al,n}$ if $n\leq 0$;
\item $A_{x\vec{y}^J\scriptstyle [1,{i_\ell}-1]}\in H^+_{\al,n}$ and $A_{x\vec{y}^J\scriptstyle [1,{i_\ell}-1]}s_{{i_\ell}}\in H^-_{\al,n}$ if $n>0$.
\eem
We fix a reduced expression $\vec{t}_\tau$ of $t_\tau$ and we set 
$$\fI_{\la,v,\tau}=\fI_{t^\ast_{\la}w_0v,\vec{t}_\tau}\quand\fI^{\max}_{\la,v,\tau}=\{J\in \fI_{\la,v,\tau}\mid \deg(\xi_J)=L(v)\}$$  so that according to (\ref{firstapprox}) and the fact that the leading term of $\xi_J$ is $\sq^{L(v)}$ we have
\begin{align*}
 C_{ w_0t_\la} \sh_{\tau} &\equiv T_{w_0t_{\la+\tau}}+\sum_{v\in W_0\backslash\{\id\}}  \bq^{-L(v)}  T_{w_0t_\la v}T_{t_\tau}\equiv \sum_{v\in W_0}\sum_{J\in \fI^{\max}_{\la,v,\tau}}  T_{w_0t_\la vt_\tau^J}\mod \cH_{<0}.
\end{align*}
and 
 \begin{align*}
 C_{w_0t_\la}{\sh}_\tau =\sum_{v\in W_0}\sum_{J\in \fI^{\max}_{\la,v,\tau}} C_{w_0t_\la vt_\tau^J}.
\end{align*}
%

\section{Description of $\fI^{\max}_{\la,v,\tau}$ in terms of admissible subsets}
\label{section5}

Once and for all in this section, we fix $\tau\in P^+$ and a reduced expression $\vec{t}_\tau=s_1\ldots s_n a$ where $a\in \Pi$ and $s_i\in S$. Let  $(\be_1,\ldots,\be_n)\in\left(\Phi^+\right)^n$ and $(N_1,\ldots,N_k)\in \nN^n$ be such that the unique hyperplane separating 
$A_{s_1\ldots s_{k-1}}$ and $A_{s_1\ldots s_k}$ is $H_{\be_k,N_k}$. Following \cite{LP2}, we now introduce the concept of admissible subsets.

\begin{Def}
\label{admissible}
A subset $J=\{i_1,\ldots,i_p\}$ of $\{1,\ldots,n\}$ will be called an admissible subset if 
$$\id < s_{\be_{i_p}}< s_{\be_{i_p}}s_{\be_{i_{p-1}}}<\ldots< s_{\be_{i_p}}s_{\be_{i_{p-1}}}\ldots  s_{\be_{i_1}}$$
 is a saturated chain in the Bruhat order on $W_0$. We set $v_J:=  s_{\be_{i_p}}s_{\be_{i_{p-1}}}\ldots  s_{\be_{i_1}}$.
\end{Def}
Saying that the chain is saturated in the Bruhat order on $W_0$ is equivalent to say that $\ell(s_{\be_{i_p}} \ldots s_{\be_{i_{k}}})=p-k+1$ for all $1\leq k\leq p$. We note that if $\{i_1,\ldots,i_p\}$ is an admissible subset then so is $\{i_\ell,\ldots,i_p\}$ for all $\ell\leq p$ and we denote this subset by $J_{\ell-1}$ so that $J_{0}=J$ and $J_p=\varnothing$. Then we have $v_{J_{\ell-1}}= s_{\be_{i_{p}}}\ldots  s_{\be_{i_\ell}}$.

\begin{Exa}
\label{admissible-subset}
Let $W$ be of type $\tG_2$ as in Example \ref{Exa-G2-1} and let $\tau=2\al_1+3\al_2\in P^+$. We fix the following reduced expression
$$\vec{t}_\tau=s_{\al_0}s_{\al_2}s_{\al_1}s_{\al_2}s_{\al_0}s_{\al_2}s_{\al_1}s_{\al_2}s_{\al_1}s_{\al_2}.$$
The sequence of roots $(\be_1,\ldots,\be_{10})$ associated to $\vec{t}_\tau$ is
{\footnotesize $$ (\al_1+2\al_2,\al_1+\al_2,2\al_1+3\al_2,\al_1+2\al_2,\al_1+\al_2,\al_1+3\al_2,\al_1+2\al_2,2\al_1+3\al_2,\al_1+\al_2,\al_1).$$}
Following \cite[Example 10.12]{LP2}, we know that there are 14 admissible subsets and we describe these sets in the table below. In the column saturated chains, we only put the extremal element and one can recover the full chain by adding to the chain all the elements above in the same column: 
for instance, the saturated chain associated to the admissible subset~$\{3,9,10\}$ is $\id<s_{\al_1} <s_{\al_1} s_{\al_1+\al_2}<s_{\al_1}s_{\al_1+\al_2} s_{2\al_1+3\al_2}$.

\medskip

\renewcommand{\arraystretch}{1.2}

{\tiny $$\begin{array}{|c|c|c|ccccc}\hline
\text{Saturated chains} & \text{Reduced expression}& \text{admissible subset}\\\hline
1 & 1&\emptyset\\\hline
 s_{\al_1}&s_{\al_1} & \{10\}\\\hline
s_{\al_1}s_{\al_1+\al_2}& s_{\al_2}s_{\al_1} &\{9,10\},\{5,10\},\{2,10\} \\\hline
 s_{\al_1}s_{\al_1+\al_2}s_{2\al_1+3\al_2}&s_{\al_1}s_{\al_2}s_{\al_1}& \{8,9,10\}, \{3,9,10\},\{3,5,10\}\\\hline
 & & \{7,8,9,10\},\{4,8,9,10\},\\
s_{\al_1}s_{\al_1+\al_2}s_{2\al_1+3\al_2}s_{\al_1+2\al_2}&s_{\al_2}s_{\al_1}s_{\al_2}s_{\al_1}& \{1,8,9,10\},\{1,3,9,10\}, \\
& &\{1,3,5,10\}\\
\hline
s_{\al_1}s_{\al_1+\al_2}s_{2\al_1+3\al_2}s_{\al_1+2\al_2}s_{\al_1+3\al_2}&s_{\al_1}s_{\al_2}s_{\al_1}s_{\al_2}s_{\al_1}&\{6,7,8,9,10\}\\\hline
\end{array}$$}

\end{Exa}
\begin{Rem}
\label{LP-admissible}
Our definition of admissible subset is slightly different than the one in \cite{LP2} where they work with reduced expressions of $t_{-\tau}$. The connection between those two definitions will be made clear in the next section.
\end{Rem}

\medskip

Recall the definition of $L(\be)$ ($\be\in\Phi^+$) and $\cC_{v}$ ($v\in W_0$) in Section \ref{Lpoints} and \ref{quarters} respectively. 
\begin{Def}
\label{domset}
Let $J=\{i_1,\ldots,i_p\}$ be an admissible subset. We say that $J$ is 
\ben
\item $\la$-dominant for $\la\in P^+$ if the $p_J(t_\la v;\vec{t}_\tau)\subset \cC_{\id}$,
\item maximal if $L(s_{i_\ell})=L(\be_{i_\ell})$ for all $1\leq \ell\leq p$.
\een
\end{Def}
Recall that we always assume that $0$ is an $L$-weight so that $L(v_{J})=\sum^p_{k=1} L(\be_{i_{k}})$ for all $J$. In particular,  if $J\in\fI^{\max}_{\la,v,\tau}$ then $J$ must be maximal in order to satisfy $\deg(\xi_J)=L(v_J)$.

\medskip

We are now ready to state the main result of this paper. Recall the notations introduced at the end of the previous section.

\begin{Th}
\label{main}
Let $\la\in P^+$ and $v\in W_0$. 
\ben
\item If $J$ is admissible, $\la$-dominant and maximal  then $J\in\fI^{\max}_{\la,v_J,\tau}$.
\item If $J\in \fI^{\max}_{\la,v,\tau}$ then $v=v_J$, $J$ is admissible, $\la$-dominant and maximal. 
\een
\end{Th}

The rest of this section is devoted to the proof of this theorem.

\begin{Prop}
\label{ksimp}
Let $v\in W_0$, $\la\in P$ and fix $k\in \{1,\ldots,n\}$. 
\ben
\item The hyperplane separating the alcoves $t_{\la}v \ta{}{1}{k-1}A_0$ and $t_{\la}v \ta{}{1}{k}A_0$ is $H_k:=H_{v\be_{k},N_{k}+\scal{\la}{{ v\be_{k}}^\vee}}$.
\item We have $t_\la   v \ta{}{1}{k-1}A_0=t_{\la'}   v' \ta{}{1}{k}A_0\text{ where } v'=vs_{\be_k} \text{ and } \la'=s_{H_k}\la$ where $s_{H_k}$ denotes the affine reflection with respect to $H_k$.
\een
\end{Prop}
\begin{proof}
The hyperplane separating $v\ta{}{1}{k-1}A_0$ and $v\ta{}{1}{k}A_0$ is $vH_{\be_{k},N_{k}}=H_{v\be_{k},N_{k}}$.
Hence the hyperplane $H_k=H_{v\be_{k},N_{k}+\scal{\la}{v\be_{k}^\vee}}$ separates the two alcoves $t_{\la}v \ta{}{1}{k-1}A_0$ and $t_{\la}v \ta{}{1}{k}A_0$.
We have
\begin{align*}
t_\la  v \ta{}{1}{k-1}A_0&=t_\la  v \ta{}{1}{k}s_kA_0\\
&=s_{H_k}t_\la  v \ta{}{1}{k}A_0\\
&=t_{s_{H_k}\la} s_{v\be_{k}} v \ta{}{1}{k}A_0\\
&=t_{s_{H_k}\la} vs_{\be_k}  \ta{}{1}{k}A_0.
\end{align*}
\end{proof}

Let $(\la,v)\in P\times W_0$ and $J=\{i_1,\ldots,i_p\}\subset \{1,\ldots,n\}$.  For all $0\leq \ell\leq p$ we define  the elements $v_\ell, \la_\ell$ and $J_\ell$ by 
\bem
\item $v_0=v$ and $v_\ell=v_{\ell-1}s_{\be_{i_{\ell}}}$;
\item $\la_0=\la$ and $\la_\ell=s_{H_\ell}\la_{\ell-1}\text{ where } H_\ell:=H_{v_{\ell-1}\be_{i_\ell},N_{i_\ell}+\scal{\la}{{ v_{i_\ell-1}\be_{i_\ell}}^\vee}}$;
\item $J_\ell:=\{i_{\ell+1},\ldots,i_p\}$.
\eem
Note that the path $p_{J_\ell}(t_{\la_\ell}v_\ell;\vec{t}_\tau)$ folds exactly $p-\ell$ times and the first fold occurs at position $i_{\ell+1}$ on the hyperplane~$H_{\ell+1}$. By a straightforward induction using Proposition \ref{ksimp}, we see that we have  for all $1\leq \ell\leq p$
$$t_{\la}v\ta{J}{1}{i_\ell-1}A_{0}=t_{\la_\ell}v s_{\be_{i_1}}\ldots s_{\be_{i_\ell}} \ta{}{1}{i_\ell}A_0=t_{\la_\ell} v_{\ell} \ta{}{1}{i_\ell}A_0.$$ 
In the case where $J$ is an admissible subset and $v=v_J$ we have $t_{\la}v\ta{J}{1}{i_\ell-1}A_{0}=t_{\la_\ell} v_{J_\ell} \ta{}{1}{i_\ell}A_0$. Further since~$J_p=\varnothing$, the path $p_{J_p}(t_\la v;\vec{t}_\tau)$ does not fold.

\begin{Exa}
\label{sequences-G2}
Let $W$ be of type $\tG_2$ as in Example \ref{Exa-G2-1} and \ref{admissible-subset}. Let $\tau=2\al_1+3\al_2\in P^+$ and fix the reduced expression $t_\tau=s_{\al_0}s_{\al_2}s_{\al_1}s_{\al_2}s_{\al_0}s_{\al_1}s_{\al_2}s_{\al_1}s_{\al_2}s_{\al_1}$.
Let $J:=\{3,5,10\}$, $\la\in P$ and $v=s_{\al_1}s_{\al_2}s_{\al_1}\in W_0$.  In Figure \ref{firstpaths-1} we describe the paths $p_{J_\ell}(t_{\la_\ell} v_\ell; \vec{t}_\tau)$ for $0\leq \ell\leq 3$  and the sequence $(\la_0,\ldots,\la_3)$ obtained  in the procedure above. 
 We write $\textcircled{\scalebox{.74}{$\ell$}}$ for the alcove $A_{t_{\la_\ell}v_\ell}$ (so that the path  $p_{J_\ell}(t_{\la_\ell}v_\ell ; \vec{t}_\tau)$ starts at the alcove  $\textcircled{\scalebox{.74}{$\ell$}}$) and the light gray alcoves represent~$A_{t_{\la_\ell}}$. 

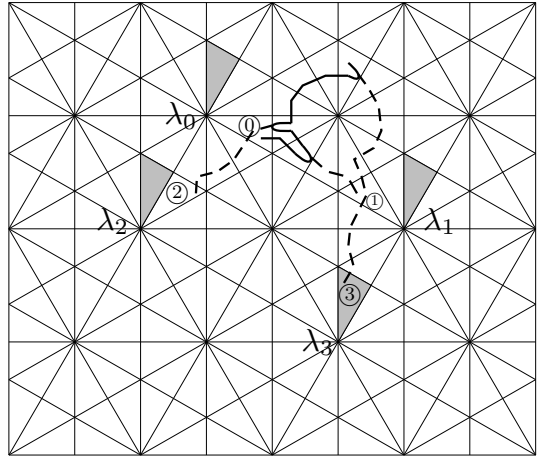
\begin{figure}[h!]
\begin{center}
\begin{pspicture}(-4.5,-3)(4.5,3)
\psset{linewidth=0.05mm}
\psline(3.464,3)(-3.464,3)
\psline(3.464,-3)(-3.464,-3)
\psline(3.464,3)(3.464,-3)
\psline(-3.464,3)(-3.464,-3)

\psline(3.464,1.5)(-3.464,1.5)
\psline(3.464,0)(-3.464,0)
\psline(3.464,-1.5)(-3.464,-1.5)

\psline(2.598,3)(2.598,-3)
\psline(1.732,3)(1.732,-3)
\psline(.866,3)(.866,-3)
\psline(0,3)(0,-3)
\psline(-2.598,3)(-2.598,-3)
\psline(-1.732,3)(-1.732,-3)
\psline(-.866,3)(-.866,-3)

\psline(1.732,3)(3.464,2)
\psline(0,3)(3.464,1)
\psline(-1.732,3)(3.464,0)
\psline(-3.464,3)(3.464,-1)
\psline(-3.464,2)(3.464,-2)
\psline(-3.464,1)(3.464,-3)
\psline(-3.464,0)(1.732,-3)
\psline(-3.464,-1)(0,-3)
\psline(-3.464,-2)(-1.732,-3)

\psline(-1.732,3)(-3.464,2)
\psline(0,3)(-3.464,1)
\psline(1.732,3)(-3.464,0)
\psline(3.464,3)(-3.464,-1)
\psline(3.464,2)(-3.464,-2)
\psline(3.464,1)(-3.464,-3)
\psline(3.464,0)(-1.732,-3)
\psline(3.464,-1)(0,-3)
\psline(3.464,-2)(1.732,-3)

\psline(-3.464,0)(-1.732,-3)
\psline(-3.464,3)(0,-3)
\psline(-1.732,3)(1.732,-3)
\psline(0,3)(3.464,-3)
\psline(1.732,3)(3.464,0)

\psline(3.464,0)(1.732,-3)
\psline(3.464,3)(0,-3)
\psline(1.732,3)(-1.732,-3)
\psline(0,3)(-3.464,-3)
\psline(-1.732,3)(-3.464,0)

\rput(-1.2,1.5){{\scalebox{1.2}{$\la_0$}}}
\pspolygon[fillstyle=solid,fillcolor=lightgray](-.866,1.5)(-.866,2.5)(-.433,2.25)
\rput(-.3,1.34){{\scalebox{.9}{$\textcircled{\footnotesize{0}}$}}}

\rput(2.2,.1){{\scalebox{1.2}{$\la_1$}}}
\pspolygon[fillstyle=solid,fillcolor=lightgray](1.732,0)(1.732,1)(2.166,.75)
\rput(1.35,.35){{\scalebox{.7}{$\textcircled{\footnotesize{1}}$}}}

\rput(-2.1,.1){{\scalebox{1.2}{$\la_2$}}}
\pspolygon[fillstyle=solid,fillcolor=lightgray](-1.732,0)(-1.732,1)(-1.3,.75)
\rput(-1.25,.45){{\scalebox{.9}{$\textcircled{\footnotesize{2}}$}}}

\rput(.6,-1.5){{\scalebox{1.2}{$\la_3$}}}
\pspolygon[fillstyle=solid,fillcolor=lightgray](.866,-1.5)(.866,-.5)(1.3,-.75)
\rput(1.02,-.9){{\scalebox{.9}{$\textcircled{\footnotesize{3}}$}}}

\psline[linewidth=.3mm](-.15,1.2)(.1,1.2)(.3,1)
\pscurve[linewidth=.3mm](.3,1)(.5,.9)(.4,1.1)
\psline[linewidth=.3mm](.4,1.1)(.25,1.3)
\pscurve[linewidth=.3mm](.25,1.3)(0,1.35)(.25,1.4)
\psline[linewidth=.3mm](.25,1.4)(.25,1.7)(.4,1.9)(.7,2.03)(1,2.03)
\pscurve[linewidth=.3mm](1,2.03)(1.13,2)(1,2.2)

\psline[linewidth=.3mm,linestyle=dashed,linecolor=black](1.114,0.47)(0.99,0.69)(0.72,.76)(.51,.95)

\psline[linewidth=.3mm,linestyle=dashed,linecolor=black](-1.004,0.47)(-0.99,0.69)(-0.72,.76)(-.51,.91)(-.25,1.3)(0,1.37)

\psline[linewidth=.3mm,linestyle=dashed,linecolor=black](0.943,-.718)(1.071,-.5)(1,-.233)(1.027,0.022)(1.236,.442)
(1.175,0.692)(1.079,0.921)(1.336,1.073)(1.435,1.302)(1.4,1.625)(1.255,1.882)(1.13,2.05)

\end{pspicture}
\end{center}
\caption{Sequences associated to the set $J=\{3,5,10\}$.}
\label{firstpaths-1}
\end{figure}
 \end{Exa}

We are now ready to prove the first part of Theorem \ref{main}.

\begin{Prop}
Let $\la\in P^+$. If $J$ is admissible, $\la$-dominant and maximal then $J\in \fI^{\max}_{\la,v_J,t_\tau}$.
\end{Prop}

\begin{proof}
In order to prove that $J=\{i_1,\ldots,i_p\}\in J\in \fI^{\max}_{\la,v_J,t_\tau}$, since $J$ is maximal,  we need to show that 
$$w_0t_{\la} v_J\ta{J}{1}{i_\ell-1}s_{i_\ell}<w_0t_{\la} v_J\ta{J}{1}{i_\ell-1} \text{ for all $1\leq \ell\leq p$}.$$
We have
\ben
\item $w_0t_{\la} v_J\ta{J}{1}{i_\ell-1} =w_0t_{\la_{\ell-1}}v_{J_{\ell-1}} \ta{}{1}{i_\ell-1}$;
\item $v_{J_{\ell-1}}=s_{\be_{i_p}}\ldots s_{\be_{i_\ell}}$ so that $v_{J_{\ell-1}}\be_{i_\ell}\in \Phi^{-}$ and $w_0v_{J_{\ell-1}}\be_{i_\ell}\in \Phi^{+}$;
 \item the hyperplane separating 
 $$w_0t_{\la_{\ell-1}} v_{J_{\ell-1}} \ta{}{1}{i_\ell-1}A_0\quand w_0t_{\la_{\ell-1}} v_{J_{\ell-1}} \ta{}{1}{i_\ell-1}s_{i_\ell}A_0$$
  is equal to $H_{w_0v_{J_{\ell-1}}\be_{i_{\ell}},m}$ where $m<0$ since $J$ is $\la$-dominant.
\een
We have $v^{-1}_{J_{\ell-1}}w_0w_0v_{J_{\ell-1}}\be_{i_{\ell}}=\be_{i_{\ell}}\in \Phi^+$ which implies that the 
 quarter $\cC_{\la^\ast_{\ell-1},w_0v_{J_{\ell-1}}}$  is oriented toward $+\infty$ in the direction $w_0v_{J_{\ell-1}}\be_{i_{\ell}}$. It follows that 
\begin{align*}
&w_0t_{\la_{\ell-1}} v_{J_{\ell-1}} \ta{}{1}{i_\ell-1}A_0\in H_{w_0v_{J_{\ell-1}}\be_{i_{\ell}},m}^-\quand\\
&w_0t_{\la_{\ell-1}} v_{J_{\ell-1}} \ta{}{1}{i_\ell-1}s_{i_\ell}A_0\in H_{w_0v_{J_{\ell-1}}\be_{i_{\ell}},m}^+
\end{align*}
hence the result by Proposition \ref{basics} since $m<0$. \end{proof}

We now focus on the second part. We start by proving some technical lemmas. Recall the definitions of $A[\al]<n$ and $A[\al]<A'[\al]$ at the end of  Section \ref{geometric} and of $\sH_{x,y}$ for $x,y\in W_e$ in Section \ref{notation-standard}.

\begin{Lem}
Let $\la\in P^+$ and $v\in W_0$. We have $\ov{\sH_{\tlas v,t_\tau}}\subset \{\de\in \Phi^+\mid v^{-1}w_0\de\in \Phi^+\}$. \end{Lem}
\begin{proof}
Let $\de\in\Phi^+$ be such that  
$$H_{\de,N}\in \sH_{\tlas v,t_\tau}=H(A_0,A_{\tlas v})\cap H(A_{\tlas v},A_{\tlas v t_\tau}).$$ First, since~$\la\in P^+$ we have $A_{\tlas v}[\de]<1$. Next since $H_{\de,N}\in H(A_0,A_{\tlas v})$, it follows that $A_{\tlas v}[\de]<N\leq 0$.
If the quarter $\cC_{\la^\ast,w_0v}$ is oriented towards $-\infty$ in the direction $\de$ then $N>A_{\tlas v}[\de]\geq  A_{\tlas v t_\tau}[\de]$ but in this case we cannot have~$H_{\de,N}\in H(A_{\tlas v},A_{\tlas vt_\tau})$. This shows that  the quarter $\cC_{\la^\ast,w_0v}$ is  oriented towards $+\infty$ in the direction $\de$ and thus $v^{-1}w_0\de\in \Phi^+$ as required. 
\end{proof}

We remark that if $\la\in P^+$ and $A_{\tlas v}\notin \cC_{w_0}$ then $A_{t_\la v}\notin \cC_{\id}$ and there exists a simple root~$\al_i\in \Phi^+$ such that~$v^{-1}\al_i\in \Phi^-$ and $-1<A_{t_\la v}[\al_i]<0$. Equivalently $-w_0\al_i\in \Phi^+$ and $0<A_{\tla v}[-w_0\al_i]<1$. This implies, considering the positive root $-w_0\al_i\in \Phi^+$, that the inclusion
$$\ov{\sH_{\tlas v,t_\tau}}\subset \{\de\in \Phi^+\mid v^{-1}w_0\de\in \Phi^+\}$$
 is strict  and $\fI^{\max}_{\la,v,\tau}$ has to be empty by Theorem \ref{bound}.

\medskip

The next lemma generalises the idea above and gives some restrictions for the set $\fI^{\max}_{\la,v,\tau}$  to be non-empty.


\begin{Lem}
\label{condition-incone}
Let $\la\in P$ and $v\in W_0$.  Let $k\in \nN$ be such that  $A_{\tlas v \ta{}{1}{k}}\in \cC_{w_0}$ and let $J=\{i_1,\ldots,i_p\}\in \fI_{\la,v,\tau}$  be such that $i_1>k$. We have
\ben
\item $\ov{\sH_{\tlas v\ta{}{1}{k},\ta{}{k+1}{n}}}\subset \{\de\in \Phi^+\mid v^{-1}w_0\de\in \Phi^+\}$;
\item if $A_{\tlas v\ta{}{1}{i_1-1}}\notin \cC_{w_0}$ then $\displaystyle \xi_J=\sum_{k=1}^p L(s_{i_k})<L(v)$.
\een
\end{Lem}

\begin{proof}
Let $\de\in\Phi^+$ be such that $H_{\de,N}\in \sH_{\tlas v\ta{}{1}{k},\ta{}{k+1}{n}}$.
Assume first that the quarter $\cC_{\la^\ast,w_0v}$ is oriented towards~$-\infty$ in the direction $\de$. Then we must have 
$$A_{\tlas v}[\de]\geq \underset{<0}{\underbrace{A_{\tlas v\ta{}{1}{k}}[\de]}}\geq A_{\tlas v\ta{}{1}{k+1}}[\de]\geq A_{\tlas vt_\tau}[\de].$$
But $H_{\de,N}\in H(A_0,A_{\tlas v\ta{}{1}{k}})$ implies that $0\geq N>A_{\tlas v\ta{}{1}{k}}[\de]$. Therefore in this case we cannot have~$H_{\de,N}\in H(A_{\tlas v\ta{}{1}{k+1}},A_{\tlas vt_\tau})$. This shows that the quarter $\cC_{\la^\ast,w_0v}$ has to be oriented towards $+\infty$ in the direction $\de$ that is $v^{-1}w_0\de\in \Phi^+$. 

\medskip

We prove (2). Since $J\in \fI_{\la,v,\tau}$, there is a term of degree $\sum_{k=1}^p L(s_{i_k})$ that appear in the product $T_{t^\ast_{\la}w_0v_J}T_{t_\tau}$. Further, for all $k$ such that $k<i_1$, there is also a term of degree $\sum_{k=1}^p L(s_{i_k})$ in product $T_{t^\ast_{\la}w_0v_J\ta{}{1}{k}}T_{\ta{}{k+1}{n}}$. Let $k_0<i_1$ be the index such that $A_{\tlas v\ta{}{1}{k_0-1}}\in \cC_{w_0}$ and $A_{\tlas v\ta{}{1}{k_0}}\notin \cC_{w_0}$ and let $\al\in \De$ be the direction of the hyperplane that separates those two alcoves. The quarter $\cC_{\la^\ast,w_0v}$ has to be oriented toward~$+\infty$ in the direction $\al$ so that $v^{-1}w_0\al\in \Phi^+$. Next 
$$\ov{\sH_{\tlas v \ta{}{1}{k_0},\ta{}{k_0+1}{n}}}=\explain{\subset \{\de\in \Phi^+\mid v^{-1}w_0\de\in \Phi^+\}}{\ov{\sH_{\tlas v \ta{}{1}{k_0-1},\ta{}{k_0}{n}}}}-\{\al\}.$$
Theorem \ref{bound} now implies that the maximal degree that can appear in  the product $T_{t^\ast_{\la}w_0v_J\ta{}{1}{k_0}}T_{\ta{}{k_0+1}{n}}$ is strictly less than $L(v)$. But there is a term of degree $\sum_{k=1}^p L(s_{i_k})$, hence the result.

\end{proof}

\begin{Lem}
\label{simp}
Let $v\in W_0$, $\la\in P$ and fix $k\in \{1,\ldots,n\}$ such that 
$$\tlas  v \ta{}{1}{k-1}s_{k}<\tlas  v \ta{}{1}{k-1}\quand A_{\tlas  v \ta{}{1}{k-1}}\in \cC_{w_0}.$$ 
Then  $w_0v\be_{k}\in \Phi^+$. In particular $w_0vs_{\be_k}>w_0v$ and $v s_{\be_k}<v$.
\end{Lem}
\begin{proof}
The hyperplane $H$ that separates $A_{\tlas  v \ta{}{1}{k-1}}$ and $A_{\tlas  v \ta{}{1}{k-1}s_k}$ is $H=H_{w_0v\be_{k},N_{k}+\scal{\la^\ast}{{w_0v\be_{k}}^\vee}}$. 
Since $A_{\tlas  v \ta{}{1}{k-1}}\in \cC_{w_0}$, we must have 
$$A_{\tlas  v \ta{}{1}{k-1}}\in H^-\quand A_{\tlas  v \ta{}{1}{k-1}s_k}\in H^+,$$ which means that the quarter $\cC_{\la^\ast,w_0v}$ has to be oriented toward $+\infty$ 
in the direction $|w_0v\be_{k}|\in \Phi^+$, where $|w_0v\be_{k}|$ denotes the unique positive root colinear to $w_0v\be_{k}$. According to Lemma \ref{pminfty}, since $w_0v^{-1}w_0v\be_{k}=\be_k\in \Phi^+$, we have~$w_0v\be_{k}\in \Phi^+$.
\end{proof}

We are now ready to prove the second statement of Theorem \ref{main}. Let $\la\in P^+$, $v\in W_0$ and assume that $J=\{i_1,\ldots,i_p\}\in \fI^{\max}_{\la,v,\tau}$. We need to show that $J$ is admissible, $\la$-dominant and that $v=v_J$.

\medskip

By maximality of $J$, we have $\sum_{k=1}^p L(s_{i_k})=L(v)$ and the second statement of Lemma \ref{condition-incone} implies
$$A_{\tlas v\ta{}{1}{i_1-1}}= A_{t_{\la_1}^\ast w_0v_1\ta{}{1}{i_1}}\in \cC_{w_0}.$$
Note also that $J_1\in  \fI_{\la_1,v_1,t_\tau}$. Applying Lemma \ref{condition-incone} to $J_1$ and $i_2>i_1$ we get
$$\ov{\sH_{t_{\la_1}^\ast w_0v_1\ta{}{1}{i_1},\ta{}{i_1+1}{n}}}\subset \{\de\in \Phi^+\mid v_1^{-1}w_0\de\in \Phi^+\}.$$
Then Theorem \ref{bound} implies that
$$\sum_{k=2}^n L(s_{i_k})\leq L(v_1).$$
But $v_1=vs_{\be_{i_1}}$ which is strictly less than $v$ by Lemma \ref{simp}. From there we see that $L(v_1)\leq L(v)-L(\be_1)=L(v)-L(s_{i_1})$ the last equality coming from the fact that~$J$ is maximal. Putting everything together, we get 
$$L(v)-L(s_{i_1})=\sum_{k=2}^n L(s_{i_k})\leq L(v_1)\leq L(v)-L(s_{i_1}).$$
It follows that we have equality throughout. Further, $\deg(\xi_{J_1})=L(v_1)$ also implies that $A_{t_{\la_1}^\ast w_0v_1\ta{}{1}{i_2-1}}\in \cC_{w_0}$ by Lemma \ref{condition-incone}. At this stage, we have 
$$J_2\in \fI_{\la_2,v_2,t_\tau}\quand A_{t_{\la_1}^\ast w_0v_1\ta{}{1}{i_2-1}}=A_{t_{\la_2}^\ast w_0v_2\ta{}{1}{i_2}}\in \cC_{w_0}.$$
We can therefore argue by induction to show that for all $\ell\in \{1,\ldots,p\}$ we have 
$$A_{t_{\la_\ell}^\ast w_0v_\ell\ta{}{1}{i_\ell}}\in \cC_{w_0}\quand L(v_\ell)=L(v)-\sum_{k=1}^\ell L(s_{i_k}).$$
In particular, we have $v_p=\id$. But $v_p=vs_{\be_{i_1}}\ldots s_{\be_{i_p}}$ so $v=s_{\be_{i_p}}\ldots s_{\be_{i_1}}=v_J$. 

\medskip

We have also seen that for all $1\leq \ell\leq p-1$
$$\explain{=A_{t_{\la}^\ast w_0v\ta{J}{1}{i_\ell}}}{A_{t_{\la_\ell}^\ast w_0v_\ell\ta{}{1}{i_\ell}}}\in \cC_{w_0}\quand \explain{=A_{t_{\la}^\ast w_0v\ta{J}{1}{i_{\ell+1}-1}}}{A_{t_{\la_\ell}^\ast w_0v_\ell\ta{}{1}{i_{\ell+1}-1}}}\in \cC_{w_0}$$
hence showing that for all $k\leq i_p$ we have $A_{t_{\la}^\ast w_0v\ta{J}{1}{k}}\in \cC_{w_0}$. For the last values of $k$ we simply note that 
$$A_{t_{\la}^\ast w_0v\ta{J}{1}{i_p}}=A_{t_{\la_p}^\ast w_0\ta{}{1}{i_p}}\in \cC_{w_0}\ (\text{since $v_p=\id$})$$
and since $\tau\in P^+$ we must have $A_{t_{\la_p}^\ast w_0\ta{}{1}{k}}\in \cC_{w_0}$ for all $k\geq i_p$. It follows that $J$ is $\la$-dominant.

\medskip

It remains to show that $J$ is admissible. Applying the proof above in the case where $L$ is the usual length function $\ell$ shows that $\ell(v)=p$. We have proved that 
$$\id<vs_{\be_{i_1}}\ldots s_{\be_{i_p}}<  \ldots <v s_{\be_{i_1}}s_{\be_{i_2}}<vs_{\be_{i_1}}<v$$
which implies that
$$\id < s_{\be_{i_p}}< s_{\be_{i_{p}}}s_{\be_{i_{p-1}}}<\ldots<  s_{\be_{i_p}} \ldots s_{\be_{i_{2}}}s_{\be_{i_1}}. $$
This sequence is of length $p+1$ hence it has to be saturated.  This completes the proof of Theorem \ref{main}.

\begin{Cor}
\label{adm-sum}
Let $\la\in P^{+}$ and $\tau \in P^+$. We have 
$$C_{w_0t_\la}\sfs_\tau= C_{w_0t_\la}\sh_{\tau}=\sum_{J} C_{w_0t_\la v_Jt_{\tau}^J}$$
where the sum is taken over all  $\la$-dominant maximal admissible subset $J$. 
\end{Cor}

\begin{Rem}
If we are working in the case of equal parameters or when $W$ is not of type $\tC_n$ or $\tA_1$, we can remove the condition of maximality since all parallel hyperplanes have same weights and therefore all admissible subsets are maximal. 
\end{Rem}

\medskip

Any element in $W_e$ can be uniquely written under the form $t_{\wt(w)}\theta(w)$ where $\wt(w)\in P$ is called the weight of $w$ and $\theta(w)\in W_0$. In this section, we have shown that $\theta(v_Jt_{\tau}^J)=\id$ so that if we set~$\mu(J):=\wt(v_Jt_{\tau}^J)$, Corollary \ref{adm-sum} now reads
$$C_{w_0t_\la}\sfs_\tau= C_{w_0t_\la}\sh_{\tau}=\sum_{J} C_{w_0t_{\la+\mu(J)}}$$
where the sum is taken over all  $\la$-dominant maximal admissible subset $J$.

\begin{Rem}
Gaussent and Littelmann \cite{G-L:05} have shown that the Littelmann path model is equivalent to the gallery model and they expressed $m_{\la,\tau}^\mu$ as the number of certain galleries. In \cite{S:06}, Schwer has computed the structure constants of the MacDonald spherical functions $P_\la(\sq^{-1})$ using galleries. The functions $P_\la(\sq^{-1})$ interpolate the Weyl characters and the monomial functions as we have $P_\la(0)=\sfs_\la$ and $P_\la(1)=m_\la$. Schwer then obtained a combinatorial description of $m_{\la,\tau}^\mu$ by describing which of the coefficients in the expansion~$P_{\la}(\bq^{-1})P_{\tau}(\bq^{-1})$ survives the specialisation $\bq^{-1}=0$. The result of Schwer is indeed equivalent to our result but its proof is quite different : it uses the link between the combinatorics of positive folded galleries of Gaussent and  Littelmann and the action of the Bernstein basis on the periodic Hecke module of Lusztig~\cite{Lus1}. Another approach to computing structure constants of MacDonald spherical functions can be found in \cite{Par:06} where Parkinson gives a formula  in terms of intersection cardinalities in affine buildings.
\end{Rem}
%
%
%
%
%

\section{Admissible subsets and Littelmann paths}
In this section we compare our result with the result of Lenart and Postnikov and then explain following~\cite[\S 9]{LP2} how admissible subsets give rise to Lakshmibai-Seshadri paths (which are special kind of Littelmann's paths).

\subsection{Lenart and Postnikov result}

In \cite{LP,LP2}, Lenart and Postnikov introduced the notion of admissible subset in order to compute the product $\sfs_\tau\sfs_\mu$ where $\la,\mu\in P^+$, see Theorem \ref{LP-th}. In this section we explain the connection between our result. 
We start by recalling their definition of admissible subset.  To distinguish between those two definitions, we will say that a subset is LP-admissible if it is admissible in the sense of \cite[Definition 6.1]{LP2}.

\medskip

Let $t_\tau=s_1\ldots s_n a$ be a fixed reduced expression and denote by $H_{\be_i,N_i}$ the hyperplane separating the alcoves $A_{s_1\ldots s_{i-1}}$ and $A_{s_1\ldots s_{i-1}s_i}$. Let~$s'_1\ldots s'_n a^{-1}$ be the reduced expression of $t_{-\tau}$ obtained by inverting the reduced expression above and moving $a^{-1}$ to the left. Then, the hyperplane separating  $A_{s'_1\ldots s'_{i-1}}$ and $A_{s'_1\ldots s'_{i-1}s'_i}$ is $H_{\be'_i,N'_i}$ where $\be'_i=\be_{n-i+1}$ and $N'_i=N_{n-i+1}-\scal{\tau}{\be_{n-i+1}^{\vee}}$. A set $J:=\{j_1,\ldots,j_p\}$ is said to be LP-admissible if the sequence 
$$\id<s_{\be'_{i_1}}<s_{\be'_{i_1}}s_{\be'_{i_2}}<\ldots<s_{\be'_{i_1}}s_{\be'_{i_2}}\ldots s_{\be'_{i_p}}$$
is a saturated sequence in the Bruhat order in $W_0$. 

\begin{Rem}
\ben
\item 
The involution  $i\mapsto n-i+1$ of $\{1,\ldots,n\}$ induces a bijection  $J\mapsto J^\dag$ between admissible subsets in the sense of Definition \ref{admissible} and admissible subsets in the sense of Lenart and Postnikov. 
 \item In \cite[Definition 5.2]{LP2}, the weight of an LP-admissible subset  is defined to be $\wt(t_{-\tau}^{J^\dag})$. The equality~$t_{-\tau}^{J^\dag}(t_{\tau}^J)A_0=A_0$  implies that  $\wt(t_{-\tau}^{J^\dag})=-\wt(v_Jt_{\tau}^J)=-\mu(J)$ where the function $\mu$ is defined at the end of the previous section. We set $\mu(J^\dag)=\mu(J)$. \een
 \end{Rem}

\medskip 
 
Before being able to state the result of Lenart and Postnikov, we need to introduce some more data associated to the LP-admissible subset $J^\dag$.  

\medskip

Let $H_{-\ga_i,\ell_i}$ where $\ga_i\in \Phi^+$ be the hyperplane that separates the alcoves $A_{\tam{J^\dag}{1}{i}}$ and $A_{\tam{J^\dag}{1}{i+1}}$ for all $1\leq i\leq n$. Let $\al$ be a simple root and $\ell_\infty^\al=\scal{\mu(J)}{\al^\vee}$. Following \cite{LP2}, we set  
\begin{align*}
I(J^\dag,\al)&=\{i\mid \ga_i= \al\} \\
 L(J^\dag,\al)&=\left\{\ell_{i}\mid i\in I(J^\dag,\al)\right\}\cup \{\ell_\infty^\al\}\\
 M(J^\dag,\al)&=\max L(J^\dag,\al).
 \end{align*}

\begin{Exa}
We keep the setting of Example \ref{admissible-subset}. Consider the admissible subset $J:=\{3,5,10\}$ and the corresponding LP-admissible subset~$J^\dag:=\{1,6,8\}$. 
We have $\mu(J)=\al_2$ so that $\ell_\infty^{\al_2}=2$ and $\ell_\infty^{\al_1}=-1$. Next we compute
$$I(J^\dag,\al_2)=\{2,6,10\} \quand I(J^\dag,\al_1):=\{1,8\}$$
so that 
$$L(J^\dag,\al_2)=\{0,1,2\} \quand L(J^\dag,\al_1):=\{0,-1\}.$$
Finally we obtain $M(J^\dag,\al_2)=2$ and $M(J^\dag,\al_1)=0$. In the figure below, we have drawn the path $p_{J}(\id;\vec{t}_{-\tau}^{J^\dag})$. The set $I(J^\dag,\al_i)$ indicates when the path is crossing an hyperplane of direction $\al_i$ and the integers $M(J^\dag,\al_i)$ are such that the whole path lies in $H^+_{\al_i,-M(J^\dag,\al_i)}$. Further, these are the minimal integers with this property. 
\begin{figure}[h!]
\begin{center}
\psset{linewidth=.1mm,unit=1.2cm}

\begin{pspicture}(-2.598,-1.5)(.866,1.5)

\psline(-2.598,-1.5)(.866,-1.5)
\psline(-2.598,0)(.866,0)
\psline(-2.598,1.5)(.866,1.5)

\psline(-2.598,1.5)(-2.598,-1.5)
\psline(-1.732,1.5)(-1.732,-1.5)
\psline(-.866,1.5)(-.866,-1.5)
\psline(0,1.5)(0,-1.5)
\psline(.866,1.5)(.866,-1.5)

\psline(-2.598,.5)(-.866,1.5)
\psline(-2.598,-.5)(.866,1.5)
\psline(-2.598,-1.5)(.866,.5)
\psline(-.866,-1.5)(.866,-.5)

\psline(.866,.5)(-.866,1.5)
\psline(.866,-.5)(-2.598,1.5)
\psline(.866,-1.5)(-2.598,.5)
\psline(-.866,-1.5)(-2.598,-.5)

\psline(-2.598,-1.5)(-.866,1.5)
\psline(-.866,-1.5)(.866,1.5)

\psline(-2.598,1.5)(-.866,-1.5)
\psline(-.866,1.5)(.866,-1.5)

\pspolygon[fillstyle=solid,fillcolor=lightgray](0,0)(0,1)(.433,.75)

\psline[linewidth=.3mm](-1.016,-.3)(-.766,-.3)(-.566,-.5)
\pscurve[linewidth=.3mm](-.566,-.5)(-.366,-.6)(-.466,-.4)
\psline[linewidth=.3mm](-.466,-.4)(-.616,-.2)
\pscurve[linewidth=.3mm](-.616,-.2)(-.866,-.15)(-.616,-.1)
\psline[linewidth=.3mm](-.616,-.1)(-.616,.2)(-.466,.4)(-.166,.53)(.134,.53)
\pscurve[linewidth=.3mm](.134,.53)(.264,.5)(.134,.7)

\end{pspicture}
\end{center}
\vspace{-.3cm}
\caption{The path $p_{J}(\id;\vec{t}_{-\tau}^{J^\dag})$}
\end{figure}
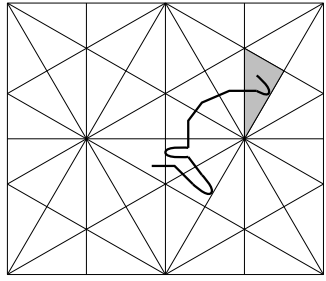
\end{Exa}

We are now ready to state Lenart and Postnikov result.
\begin{Th}[\cite{LP2}, Corollary 8.3]
\label{LP-th}
In a simple Lie algebra with Weyl group $W_0$, we have for $\la,\tau\in P^+$ 
\begin{equation}
\label{LP}
\sfs_\tau\sfs_\la=\sum_{J^\dag} \sfs_{\la+\mu(J^\dag)}
\end{equation}
where the sum is over all $LP$-admissible subsets $J^\dag$ satisfying $\scal{\la+\mu(J^\dag)}{\al^\vee}\geq M(J^\dag,\al)$ for all simple root~$\al$.
\end{Th}
We need to check that the condition on $J^\dag$ above is equivalent to $J$ being $\la$-dominant. It will be a consequence of the following  three observations:
\ben
\item  the path  $p_{J}(A_{t_\la};\vec{t}_{-\tau}^{J^\dag})$ lies in the half plane $H^+_{\al,\scal{\la}{\al^\vee}-M(J^\dag,\al)}$ and the integer $\scal{\la}{\al^\vee}-M(J^\dag,\al)$ is maximal with this property;
\item $J$ is $\la$-dominant if and only if  the path $p_{J}(t_\la v_J;\vec{t_\tau})$ lies in $\cC_{\id}$;
\item $p_{J}(t_\la v_J;\vec{t_\tau})$ is obtained from $p_{J}(A_{t_\la};\vec{t}_{-\tau}^{J^\dag})$ by translating all the alcoves by  $\mu(J)$ and reversing the order.
\een
We have
\begin{align*}
p_{J}(t_\la v_J;\vec{t_\tau})\subset \cC_\id &\eq p_{J}(t_\la v_J;\vec{t_\tau})\subset H^+_{\al,0} \text{ for all simple root $\al$}\\
&\eq \left(p_{J}(A_{t_\la};\vec{t}_{-\tau}^{J^\dag})\right)t_{\mu(J)}\subset H^+_{\al,0} \text{ for all simple root $\al$}\\
&\eq p_{J}(A_{t_\la};\vec{t}_{-\tau}^{J^\dag})\subset H^+_{\al,-\scal{\mu(J)}{\al^\vee}} \text{ for all simple root $\al$}\\
&\eq \scal{\la}{\al^\vee}-M(J^\dag,\al)\geq -\scal{\mu(J)}{\al^\vee}\text{ for all simple root $\al$}\\
&\eq \scal{\la+\mu(J)}{\al^\vee}\geq M(J^\dag,\al)\text{ for all simple root $\al$}.
\end{align*}
This theorem is equivalent to Theorem \ref{main} in the equal parameter case when there is no condition on the maximality of $J$.

\subsection{Lakshmibai-Seshadri paths}
We explain following \cite[\S 9]{LP2}, how to construct an Lakshmibai-Seshadri path (LS paths for short) from an admissible subset $J$. There is no new result in this section but this construction exhibits a strong link between LS paths and our approach in the proof of Theorem \ref{main}: compare for instance Figure \ref{firstpaths-1} and  Figure \ref{firstpaths-2}.
\medskip

 A rational $W_a$-path of shape $\mu\in P$ is a pair of sequences $(\un{v},\un{a})$ such that $\un{v}=(v_{0},v_1,\ldots,v_r)$ is a strictly decreasing chain (in the Bruhat order) of minimal length left coset representatives of $\stab_\mu(W_0)$ and $\un{a}=(a_0,a_1,\ldots,a_{r+1})$ is an  increasing sequence of rational numbers such that $a_0=0$ and $a_{r+1}=1$. 

\medskip

We identify $\pi$ with the path $\pi:[0,1]\lra V$ defined by 
$$\pi(t):=\sum_{i=1}^{j-1} (a_i-a_{i-1})v_i\mu+(t-a_{j-1})v_j\mu \text{ for $a_{j-1}\leq t\leq a_j$}.$$
The weight of $\pi$ is equal to $\pi(1)$.

\medskip
\newcommand{\vlra}{\xrightarrow{\hspace*{.5cm}}}  

A rational $W_a$-path of shape $\mu$ is called a Lakshmibai-Seshadri path if there exists a sequence of positive roots~$(\de_1,\ldots,\de_{r})$ such that 
$$v_0>v_1=v_0s_{\de_1}>v_2=v_0s_{\de_1}s_{\de_2}>\ldots>v_r=v_0s_{\de_1}\ldots s_{\de_r},$$
$\ell(v_i)=\ell(v_{i-1})-1$ and $a_i\scal{v_i\mu}{\de_i^\vee}\in \nZ$.

\begin{Rem}
Our definition of LS paths is slightly different from the one of Littelmann in \cite{Lit1} but nearly obviously equivalent. One only needs to notice that if $a=a_q=a_{q+1}=\ldots=a_{q+r}$ then the chain $(v_q,\ldots,v_{q+r})$ is an $a$-chain as defined by Littelmann.
\end{Rem}

\medskip

For the rest of this section we will need to work with a specific reduced expression of $t_\tau$. This choice is important in the proof of the third statement of Theorem \ref{LP2-9}; see \cite[\S 9]{LP2}.  Fix a total order on the set of simple roots $\{\al_1,\ldots,\al_N\}$ and write $\{\om_1,\ldots,\om_N\}$ for the corresponding fundamental weights. We define the map 
 $$\barr
\bh:& H(A_0,A_{t_\tau}) &\lra& \nR^{N+1}\\
&H_{\de,k}&\lmt& \frac{1}{\scal{\tau}{\de^\vee}}\left(k,\scal{\om_1}{\de^\vee},\ldots,\scal{\om_{N}}{\de^\vee} \right).
\ear$$
The lexicographic order on $\nR^{N+1}$ induces a total order on the set $H(A_0,A_\tau)$. Let 
$$H(A_0,A_\tau)=\{H_{\be_1,N_1},\ldots,H_{\be_n,N_n}\}$$
be such that $H_{\be_i,N_i}<H_{\be_{i+1},N_{i+1}}$ for all $i$. Then 
there exists a reduced expression of $t_\tau$ of the form $s_1\ldots s_n a$ where $a\in \Pi$, $s_i\in S$ and such that $H_{\be_i,N_i}$ is the hyperplane separating $A_{s_1\ldots s_{i-1}}$ and $A_{s_1\ldots s_i}$. \\

Let $J=\{i_1,\ldots,i_p\}$ be an admissible subset as in Definition \ref{admissible} and consider the pair $(0,v_J)\in P\times W_0$. Let~$(\la_0,\ldots,\la_{p})$, $(v_0,\ldots,v_p)$ and $(J_0,\ldots,J_p)$ be the sequences associated to $(0,v_J)$ constructed in Section \ref{section5}. Since $J$ is admissible we have $v_{\ell}=s_{\be_{i_p}}\ldots s_{\be_{i_\ell+1}}=v_{J_{\ell}}$.
We set 
\bem
\item $\tau_\ell=v_{J_\ell}\tau$ for $0\leq \ell\leq p$ ,
\item $\pi_\ell:[0,1]\lra V$ to be the straight path defined by $t\lmt  \mu_\ell+t\cdot \tau_\ell$,
\item  $a_\ell=\dfrac{N_{i_{\ell}}}{\scal{\tau }{{\be^{\vee}_{i_{\ell}}}}}$ for $1\leq \ell\leq p$.
\eem
Finally, we set $\un{v}=(v_{J_0},\ldots,v_{J_p})$ and $\un{a}=(a_0,\ldots,,a_{p+1})$ where $a_0=0$ and $a_{p+1}=1$. By construction, the path~$\pi=(\un{v},\un{a})$ coincide with the path~$\pi_\ell$ for all $a_{\ell}\leq t\leq a_{\ell+1}$. 
Then according to \cite[\S 9]{LP2} one can show the following result where once again, we assume that we are in the equal parameter case. The choice of our specific reduced expression for $t_\tau$ plays a crucial role in the proof of Statement (3). Recall the definition of $\mu(J)$ at the end of Section \ref{section5}.

\begin{Th}
\label{LP2-9}
Let $J$ be an admissible subset, $\la\in P^+$ and  $\pi=(\un{v_J},\un{a})$ be the $W_a$-rational path defined above. We have
\ben
\item  $\pi$ is a LS-path of shape $\tau$ and weight $\mu(J)$.
\item The set of LS-paths of shape $\tau$ is in bijection with the set of admissible subsets. 
\item $\pi$ is $\la$-dominant (i.e. $\la+\pi(t)\in \ov{\cC_{\id}}$ for all $t\in [0,1]$) if and only if the 
path  $p_J(t_\la v_J;t_{\tau})\subset \cC_{\id}$.
\een
\end{Th} 

As a direct consequence of this theorem and  Theorem \ref{main}, we obtain for $\la,\tau\in P^+$
$$ C_{ w_0t_\la}\sh_{\tau}=\sum_{\pi}C_{w_0 t_{\la+\pi(1)}}$$ 
where the sum is over all $\la$-dominant LS-paths of shape $\tau$.

\begin{Exa}
Let $W$ be of type $G_2$ and keep the notations of the previous examples. We compute the different elements needed in the construction of the path $\pi_J$ where $J=\{3,5,10\}$. Note that most of these have already been constructed in Example \ref{sequences-G2}.

\renewcommand{\arraystretch}{1.2}%

$$\begin{array}{|c|c|c|c|c|c|c|c|c|}\hline
\ell& J_\ell&\be_{i_\ell}  & v_{J_\ell}&\tau_\ell&N_{i_\ell}&a_{\ell}\\\hline
0&	\{3,5,10\} & &   s_{\al_1}s_{\al_2}s_{\al_1}&-\al_1&&0\\
1&	\{5,10\} &\de_4 & s_{\al_2}s_{\al_1}&\al_1&1&1/2\\
2&\{10\} &\de_5&\si_{\al_1}&\al_1+3\al_2&2&2/3\\	
3&\emptyset&\de_6 &\id&2\al_1+3\al_2&1&1\\\hline
\end{array}$$

\medskip

The LS-path  $\pi_J$ associated to $J$ is represented by a thick line in Figure \ref{firstpaths-2}. The paths $\pi_1$, $\pi_2$ and $\pi_3$ are represented by dashed lines. We see that that
\ben
\item[$\ast$] $\pi_J$ follows the direction $-\al_1$ for a time $1/2$;
\item[$\ast$] $\pi_J$ follows the direction $\al_1$ for a time $1/6$;
\item[$\ast$] $\pi_J$ follows the direction $\al_1+3\al_2$ for a time $1/3$.
\een

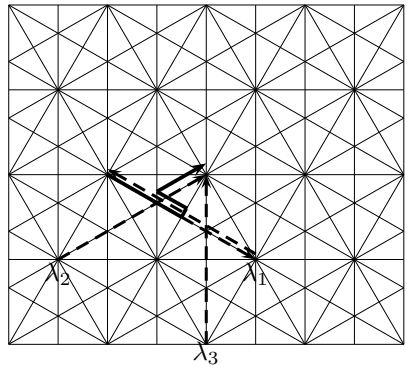
\begin{figure}[h!]
\begin{center}
\psset{unit=.75cm}
\begin{pspicture}(-4.5,-3)(4.5,3)
\psset{linewidth=0.05mm}
\psline(3.464,3)(-3.464,3)
\psline(3.464,-3)(-3.464,-3)
\psline(3.464,3)(3.464,-3)
\psline(-3.464,3)(-3.464,-3)

\psline(3.464,1.5)(-3.464,1.5)
\psline(3.464,0)(-3.464,0)
\psline(3.464,-1.5)(-3.464,-1.5)

\psline(2.598,3)(2.598,-3)
\psline(1.732,3)(1.732,-3)
\psline(.866,3)(.866,-3)
\psline(0,3)(0,-3)
\psline(-2.598,3)(-2.598,-3)
\psline(-1.732,3)(-1.732,-3)
\psline(-.866,3)(-.866,-3)

\psline(1.732,3)(3.464,2)
\psline(0,3)(3.464,1)
\psline(-1.732,3)(3.464,0)
\psline(-3.464,3)(3.464,-1)
\psline(-3.464,2)(3.464,-2)
\psline(-3.464,1)(3.464,-3)
\psline(-3.464,0)(1.732,-3)
\psline(-3.464,-1)(0,-3)
\psline(-3.464,-2)(-1.732,-3)

\psline(-1.732,3)(-3.464,2)
\psline(0,3)(-3.464,1)
\psline(1.732,3)(-3.464,0)
\psline(3.464,3)(-3.464,-1)
\psline(3.464,2)(-3.464,-2)
\psline(3.464,1)(-3.464,-3)
\psline(3.464,0)(-1.732,-3)
\psline(3.464,-1)(0,-3)
\psline(3.464,-2)(1.732,-3)

\psline(-3.464,0)(-1.732,-3)
\psline(-3.464,3)(0,-3)
\psline(-1.732,3)(1.732,-3)
\psline(0,3)(3.464,-3)
\psline(1.732,3)(3.464,0)

\psline(3.464,0)(1.732,-3)
\psline(3.464,3)(0,-3)
\psline(1.732,3)(-1.732,-3)
\psline(0,3)(-3.464,-3)
\psline(-1.732,3)(-3.464,0)

\psline[linewidth=.5mm,linecolor=black](-1.732,0)(-.433,-.75)
\psline[linewidth=.5mm,linecolor=black](-.433,-.75)(-.352,-.6)
\psline[linewidth=.5mm,linecolor=black](-.866,-.3)(-.332,-.6)
\psline[linewidth=.5mm,linecolor=black]{->}(-.866,-.3)(0,0.2)

\psline[linewidth=.4mm,linecolor=black,linestyle=dashed]{->}(-1.732,0)(.866,-1.5)
\psline[linewidth=.4mm,linecolor=black,linestyle=dashed]{<-}(-1.732,0.1)(.866,-1.4)
\psline[linewidth=.4mm,linecolor=black,linestyle=dashed]{->}(-2.598,-1.5)(0,0)
\psline[linewidth=.4mm,linecolor=black,linestyle=dashed]{->}(0,-3)(0,0)

\rput(-2.598,-1.75){\scalebox{1}{$\la_2$}}

\rput(.866,-1.75){\scalebox{1}{$\la_1$}}

\rput(0,-3.15){\scalebox{1}{$\la_3$}}

\end{pspicture}
\end{center}
\caption{LS path associated to the set $J=\{3,5,10\}$.}
\label{firstpaths-2}
\end{figure}

When doing the above procedure for all admissible subsets $J$ (described in Example~\ref{admissible-subset}) we obtain the paths described in Figure~\ref{14-LS}. 

\medskip

\begin{figure}[h!]
\begin{center}
\psset{unit=.8cm}
\begin{pspicture}(-4.5,-3)(4.5,3)
\psset{linewidth=0.05mm}

\pspolygon[fillcolor=lightgray,fillstyle=solid](0,0)(0,1)(0.433,0.75)
\psline[linewidth=.5mm,linecolor=black]{->}(0,0)(0,3)

\psline[linewidth=.5mm,linecolor=black]{->}(0,0)(2.598,1.5)

\psline[linewidth=.5mm,linecolor=black]{->}(0,0)(0,-3)

\psline[linewidth=.5mm,linecolor=black]{->}(0,0)(2.598,-1.5)
\psline[linewidth=.5mm,linecolor=black](1.3,-.75)(1.38,-.6)
\psline[linewidth=.5mm,linecolor=black]{<-}(0.1,.15)(1.4,-.6)
\psline[linewidth=.5mm,linecolor=black]{->}(.866,-.3)(1.732,0.2)

\psline[linecolor=black,linewidth=.5mm]{->}(-.866,.5)(.866,1.5)
\psline[linecolor=black,linewidth=.5mm]{->}(-1.732,1)(-.866,1.5)
\psline[linecolor=black,linewidth=.5mm]{->}(0,0)(-2.598,1.5)

\psline[linecolor=black,linewidth=.5mm]{->}(0,0)(-2.598,-1.5)
\psline[linecolor=black,linewidth=.5mm]{->}(-.866,-.5)(.866,-1.5)
\psline[linecolor=black,linewidth=.5mm]{->}(-1.732,-1)(-.866,-1.5)
\psline[linecolor=black,linewidth=.5mm](-.35,-.6)(-.433,-.75)
\psline[linecolor=black,linewidth=.5mm]{->}(-.35,-.6)(-1.732,.2)
\psline[linecolor=black,linewidth=.5mm](-.866,-.3)(-.866,-.12)
\psline[linecolor=black,linewidth=.5mm]{<-}(-.2,.2)(-.896,-.14)

\psline(3.464,3)(-3.464,3)
\psline(3.464,-3)(-3.464,-3)
\psline(3.464,3)(3.464,-3)
\psline(-3.464,3)(-3.464,-3)

\psline(3.464,1.5)(-3.464,1.5)
\psline(3.464,0)(-3.464,0)
\psline(3.464,-1.5)(-3.464,-1.5)

\psline(2.598,3)(2.598,-3)
\psline(1.732,3)(1.732,-3)
\psline(.866,3)(.866,-3)
\psline(0,3)(0,-3)
\psline(-2.598,3)(-2.598,-3)
\psline(-1.732,3)(-1.732,-3)
\psline(-.866,3)(-.866,-3)

\psline(1.732,3)(3.464,2)
\psline(0,3)(3.464,1)
\psline(-1.732,3)(3.464,0)
\psline(-3.464,3)(3.464,-1)
\psline(-3.464,2)(3.464,-2)
\psline(-3.464,1)(3.464,-3)
\psline(-3.464,0)(1.732,-3)
\psline(-3.464,-1)(0,-3)
\psline(-3.464,-2)(-1.732,-3)

\psline(-1.732,3)(-3.464,2)
\psline(0,3)(-3.464,1)
\psline(1.732,3)(-3.464,0)
\psline(3.464,3)(-3.464,-1)
\psline(3.464,2)(-3.464,-2)
\psline(3.464,1)(-3.464,-3)
\psline(3.464,0)(-1.732,-3)
\psline(3.464,-1)(0,-3)
\psline(3.464,-2)(1.732,-3)

\psline(-3.464,0)(-1.732,-3)
\psline(-3.464,3)(0,-3)
\psline(-1.732,3)(1.732,-3)
\psline(0,3)(3.464,-3)
\psline(1.732,3)(3.464,0)

\psline(3.464,0)(1.732,-3)
\psline(3.464,3)(0,-3)
\psline(1.732,3)(-1.732,-3)
\psline(0,3)(-3.464,-3)
\psline(-1.732,3)(-3.464,0)

\end{pspicture}
\end{center}
\caption{LS paths in type $G_2$.}
\label{14-LS}
\end{figure}
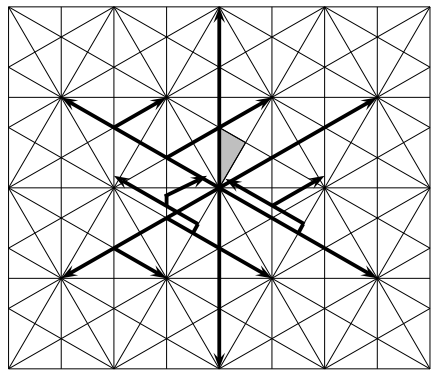
\end{Exa}

{\bf Acknowledgement:} The author would like to thank C\'edric Lecouvey and James Parkinson for many helpful discussions and for pointing out interesting references. Thanks also to the referees for suggesting many interesting modifications.

\noindent
{\small \textsc{ Jeremie Guilhot: Institut Denis Poisson, 
Universit\'e de Tours,
Universit\'e d'Orl\'eans, 
CNRS, Tours, France.}
 }\\
{\it Email address} \url{ jeremie.guilhot@idpoisson.fr}


\begin{thebibliography}{10}

\bibitem{bourbaki}
N.~Bourbaki,
\newblock {\em Groupes et Alg\`ebres de Lie, Chap 4--6}.
\newblock Hermann, Paris, 1968; Masson, Paris, 1981.

\bibitem{bremke}
K.~Bremke,
\newblock  {\em On generalized cells in affine {W}eyl groups}.
\newblock  J. of Algebra $\mathbf{191}$ (1997),149--173.

\bibitem{G-L:05}
S.~Gaussent and P.~Littelmann,
\newblock  {\em L{S} galleries, the path model, and {MV} cycles}.
\newblock  Duke Math. J. {\bf 127} (2005) ,35--88.

\bibitem{guilhot2}
J.~Guilhot,
\newblock  {\em On the lowest two-sided cell of an affine {W}eyl groups}.
\newblock  Represent. Theory {\bf 12} (2008) , 327--345.

\bibitem{guilhot7}
J.~Guilhot,
\newblock  {\em Cellularity of the lowest two-sided ideal of an affine {H}ecke
  algebra}.
\newblock Adv. Math. {\bf 255} (2014)

\bibitem{Hum}
J.~E. Humphreys,
\newblock {\em Reflection Groups and {C}oxeter Groups}.
\newblock Cambridge studies in advance mathematics {\bf 29} (1990), Cambridge
  University Press.

\bibitem{LP}
C.~Lenart and A.~Postnikov,
\newblock  {\em Affine {W}eyl groups in {K}-theory and representation theory}.
\newblock Int. Math. Res. Not. IMRN {\bf 12} (2007).

\bibitem{LP2}
C.~Lenart and A.~Postnikov,
\newblock  {\em A combinatorial model for crystals of {K}ac-{M}oody algebras}.
\newblock  Trans. Amer. Math. Soc. {\bf 360} (2008).

\bibitem{Lit1}
P.~Littelmann,
\newblock  {\em A {L}ittlewood-{R}ichardson rule for symmetrizable {K}ac-{M}oody
  algebras}.
\newblock Invent. Math. $\mathbf{116}$ (1994), no. 1-3, 329--346.

\bibitem{Lus1}
G.~Lusztig,
\newblock  {\em Hecke algebras and {J}antzen's generic decomposition patterns}.
\newblock Adv. in Math. $\mathbf{37}$ (1980), 121--164.

\bibitem{Lus15}
G.~Lusztig,
\newblock  {\em Singularities, character formulas, and a q-analog of weight
  multiplicities}.
\newblock  Analyse et Topologie sur les Espaces Singuliers (II--III),
  Asterisque 101-102 (1983), 208--227. 

\bibitem{bible}
G.~Lusztig,
\newblock {\em {H}ecke Algebras with Unequal Parameters}.
\newblock CRM Monograph Series. Amer. Math. Soc. (2003), Providence, RI.

\bibitem{R-N:03}
K.~{N}elsen and A.~{R}am,
\newblock  {\em {K}ostka-{F}oulkes polynomials and {M}acdonald spherical functions}.
\newblock Surveys in Combinatorics, London Math. Soc. Lecture Notes {\bf 307} (2003), 325--370.

\bibitem{Par:06}
J.~{P}arkinson,
\newblock  {\em {B}uildings and {H}ecke algebras}.
\newblock J. {A}lgebra (2006), 1--49.

\bibitem{S:06}
C.~Schwer,
\newblock  {\em Galleries, {H}all-{L}ittlewood polynomials, and structure constants
  of the spherical {H}ecke algebra}.
\newblock  Int. Math. Res. Not., 2006.

\bibitem{St:02}
J. Stembridge,
\newblock  {\em Combinatorial models for {W}eyl characters}.
\newblock  Adv. Math. {\bf 168} (2002), 96--131.

\bibitem{Xie:15}
X.~Xie,
\newblock  {\em A decomposition formula for the {K}azhdan-{L}usztig basis of affine
  {H}ecke algebras of rank 2}.
\newblock {\em Preprint available at arXiv:1509.05991}, 2015.

\bibitem{Xie:17}
X.~Xie,
\newblock  {\em The based ring of the lowest generalized two-sided cell of an
  extended affine {W}eyl group}.
\newblock  J. Algebra {\bf 477} (2017), 1--28.

\end{thebibliography}
\end{document}